\documentclass[reqno,final]{amsart}
\usepackage{caption}
\usepackage{subcaption}

\usepackage{orcidlink}
\usepackage{mathabx}
\usepackage[utf8]{inputenc} 
\usepackage[english]{babel}
\usepackage{amstext}
\usepackage{euscript}
\usepackage{bbm}
\usepackage{mathrsfs}
\usepackage{enumerate}
\usepackage{graphicx}
\usepackage{caption}
\usepackage{amscd}
\usepackage{verbatim}
\usepackage{amssymb}
\usepackage{amsthm}
\usepackage{amsmath}
\usepackage{amstext}
\usepackage{amsfonts}
\usepackage{latexsym}
\usepackage{mathtools} 
\usepackage{enumitem} 
\usepackage{accents} 
\usepackage[foot]{amsaddr} 
\usepackage{amssymb,amsthm,tikz,fancyhdr,cancel,enumerate,array,booktabs,setspace,pgf,tikz,fancyhdr,graphicx,color}
\usepackage[T1]{fontenc}
\usetikzlibrary{calc,angles,quotes,patterns,arrows,decorations.pathreplacing}

\definecolor{dkblue}{RGB}{1,31,91} 

\usepackage{todonotes}

\usepackage[color]{showkeys}

\numberwithin{equation}{section}
\newtheorem{theorem}{Theorem}[section]
\newtheorem{lemma}[theorem]{Lemma}
\newtheorem{corollary}[theorem]{Corollary}
\newtheorem{proposition}[theorem]{Proposition}

\newtheorem{remark}[theorem]{Remark}

\newcommand{\rd}{\mathrm{d}}

\newcommand{\bR}{\mathbb{R}}

\newcommand{\Sd}{\mathbb{S}}

\newcommand{\vs}{v_{*}}

\newcommand{\am}{a_{-}}
\newcommand{\ap}{a_{+}}
\newcommand{\xie}{\xi_{e}}
\newcommand{\p}{\varphi}
\newcommand{\K}{\mathcal{K}}

\newcommand*{\im}{\mathop{}\!\mathrm{i}}
\newcommand*{\e}{\mathop{}\!\mathrm{e}}

\newcommand{\eqdef }{\overset{\mbox{\tiny{def}}}{=}}

\begin{document}
	
	\keywords{Boltzmann Equation, Measure-Valued Solutions, Fourier Transform, Non-Cutoff Cross-Section, Inelastic Collision, Hard-Potential.}
	\subjclass[2020]{Primary 35Q20, 76P05, 35R11, 26A33; Secondary  35H20, 82B40, 82C40.  }

	\title[Inelastic Boltzmann Equation without angular cutoff]{Moments Creation for the inelastic Boltzmann equation for hard potentials without angular cutoff}
	
	
	
	
	\author[J. W. Jang]{Jin Woo Jang $^{\dagger}$ \orcidlink{0000-0002-3846-1983} }
	\address{$^{\dagger}$Department of Mathematics, Pohang University of Science and Technology (POSTECH), Pohang, South Korea. \href{mailto:jangjw@postech.ac.kr}{jangjw@postech.ac.kr} }

	\author[K. Qi]{Kunlun Qi $^{\ddagger}$ \orcidlink{0000-0003-2995-8901}}
	\address{$^{\ddagger}$School of Mathematics, University of Minnesota--Twin Cities, Minneapolis, MN 55455 USA. \href{mailto:kqi@umn.edu}{kqi@umn.edu} }
	
	
	\maketitle

	\begin{abstract}This paper is concerned with the inelastic Boltzmann equation without angular cutoff. We work in the spatially homogeneous case.  We establish the global-in-time existence of measure-valued solutions under the generic hard potential long-range interaction on the collision kernel.
	In addition, we provide a rigorous proof for the creation of polynomial moments of the measure-valued solutions, which is a special property that can only be expected from hard potential collisional cross-sections. The proofs rely crucially on the establishment of a refined Povzner-type inequality for the inelastic Boltzmann equation without angular cutoff. The class of initial data that we require is general in the sense that we only require the boundedness of $(2+\kappa)$-moment for $\kappa>0$ and do not assume any entropy bound.
	\end{abstract}

	\tableofcontents
	
	\section{Introduction}
	\label{sec:intro}
	
	\subsection{Problem statement and motivations}
	\label{sub:motivation}
	This paper is concerned with the inelastic Boltzmann equation which describes the dynamics of particles that undergo the inelastic collision process such as granular media. Though studies on the Boltzmann equation have a long history since its first appearance, the global well-posedness of the equation for inelastic collisional processes for general initial-boundary values has only been answered partially and is still highly open. In this paper, we are interested in working on the inelastic regime without any angular cutoff assumption. Especially, we are interested in the hard potential case whose range can cover the standard hard-sphere case (with an angular singularity though the angular kernel for the actual hard-sphere model has a cutoff). The motivation behind the collision cross-section with an angular singularity goes back to the very early prototype of Maxwell even before Boltzmann's formulation of the equation. Regarding the hard-potential equation without angular cutoff, we would first like to establish the global existence of measure-valued solutions in the spatially homogeneous case under a general class of initial conditions. The class of initial conditions requires a slightly higher moments bound than the energy bound, and no entropy bound is needed.   In addition, we would also like to prove the moments-creation property under the same class of initial conditions. We note that this moments-creation property is a special property that only the hard potential regime can retain. These problems have been in general considered to be very difficult due to the lack of a sufficient Povzner-type estimate for the non-cutoff inelastic Boltzmann equation, the historic development and significance of which will be presented in detail later on. Hence, we put the establishment of a Povzner-type estimate for the inelastic Boltzmann without cutoff assumption as the main goal of this paper. Then after we obtain the inelastic Povzner inequality for the non-cutoff collision kernel, we will provide the proof for the global well-posedness and the moments-creation property for the inelastic Boltzmann equation without angular cutoff.

    The establishment of the existence theory of the measure-valued solution in the hard-potential case completes our journey for the global existence of measure-valued solutions complementing our preceding work~\cite{Qi21_soft, Qi22}, in which the inelastic Boltzmann equation associated with Maxwellian molecule and soft-potential collision kernels have been studied.  Furthermore, another significant novelty of this paper is on the proof of the creation of polynomial moments of the solution obtained for the inelastic Boltzmann equation under a certain class of initial conditions. In fact, as a distinctive property of the solution for the hard-potential collision kernel (for instance, hard-sphere), the moments-creation property of the elastic Boltzmann equation has been figured out for a long period of time. We would like to introduce that Desvillettes in~\cite{Desvillettes93_moments} showed that if some moments of the initial data of more than two order are bounded, then all moments of the solution are bounded for any positive time, which essentially extended the results of Elmroth in~\cite{Elmroth83} that all moments remain bounded uniformly in time if they are initially bounded. Moreover, the result of Desvillettes has further been proved to be true by Wennberg in \cite{Wennberg94,Wennberg97} only if the energy of the initial data is bounded. Later, Morimoto-Wang-Yang showed that the same result also holds for the measure-valued solution in~\cite{morimoto2016measure} where they use the Povzner inequality by Mischler-Wennberg in~\cite{MW1999}. Meanwhile, thanks to the moments-estimates mentioned above, more intricate studies about of tail behavior of the solution have been derived, for which we refer to the seminal work~\cite{Bobylev97_moment,GPV09} by Bobylev and Gamba-Panferov-Villani as well as recent progresses~\cite{ACGM2013,Fournier2022} by Alonso-Ca\~{n}izo-Gamba-Mouhot and Fournier for an exponential moment-estimate, \cite{PT2018stretched} by Pavic-Colic-Taskovic for the propagation of stretched exponential moments, \cite{TAGP2018} by Taskovic-Alonso-Gamba-Pavlovic for the Mittag-Leffler moments.

    On the other hand, compared to the elastic case, such a moments-estimate for the inelastic Boltzmann equation is still very rare, especially under the non-cutoff assumption. In~\cite{BGP04}, Bobylev-Gamba-Panferov first established the moments-inequalities of the inelastic Boltzmann equation with cutoff hard-potential collision kernel, where a sharper version of the Povzner inequality for the inelastic collision was proved generalizing the previous elastic counterpart ~\cite{Bobylev97_moment} by Bobylev. Mischler-Mouhot-Rodriguez studied the cooling process for inelastic Boltzmann equations for hard-spheres, including the well-posedness, self-similar solution, and tail behavior in their series of papers \cite{MM2006hardsphere1,MM2006hardsphere2}.
    Besides, the pseudo-Maxwellian model, regarded as an approximated equation via the replacement of the collision kernel by a certain mean-value independent of the relative velocity, was studied in~\cite{BCG00} by Bobylev-Carrillo-Gamba; in addition, for the self-similar scaling problem, the solution with ``flat'' tail, decaying like an inverse power law, was first conjectured by Ernst-Brito in~\cite{EB02_conjecture} and was rigorously justified by Bobylev-Cercignani-Toscani in~\cite{BCT2003}. We also mention the moments equation for the inelastic model with a heat bath by Carrillo-Cercignani-Gamba in~\cite{CCG00}, Bobylev-Cercignani in~\cite{BC02_thermal} and Gamba-Panferov-Villani in~\cite{GambaDiffusively}, Alonso-Lods in \cite{AL2013} for references. For more results about the inelastic Boltzmann equation, we refer to the existence of solution near vacuum data \cite{Alonso2009}, the model with variable restitution coefficient \cite{AL2010} thanks to Alonso and his collaborator as well as the elastic limit problem \cite{MM2009inelasticlimit, MM2009inelasticlimit2} by Mischler-Mouhot. In particular, the rigorous derivation of the system of hydrodynamic equations from the inelastic Boltzmann equation has been recently presented in \cite{2008.05173} and \cite{ALT2022JSP}.

	\subsection{The inelastic Boltzmann equation}
	
	Unlike the classical elastic collision between particles, the loss of energy occurs in the impact direction during the inelastic collision process. If we let $ e $ be the restitution coefficient, the post-collisional velocities $ v', v'_{*} $ can be represented as
	\begin{equation}\label{ve}
		\left\{
		\begin{array}{lr}
			v' =  \frac{v+\vs}{2} + \frac{1-e}{4}(v-\vs) + \frac{1+e}{4}|v-\vs|\sigma, &\\
			v'_{*} = \frac{v+\vs}{2} - \frac{1-e}{4}(v-\vs) - \frac{1+e}{4}|v-\vs|\sigma,&
		\end{array}
		\right.
	\end{equation}
	where $ v, v_{*} $ are the pre-collisional velocities and $ \sigma $ is a vector on the unit sphere $ \Sd^{2} $.
	Then, without considering the dependence on a spatial variable, we will study the following spatially homogeneous Boltzmann equation in the inelastic case 
	\begin{equation}\label{IB}
		\partial_{t} f(t,v) = Q_{e}(f,f) (t,v),
	\end{equation}
	associated with the non-negative initial condition 
	\begin{equation}\label{F0}
		f(0,v) = F_{0}(v).
	\end{equation}
Here $ f=f(t,v) $ is an unknown density function of a probability density function, and $  Q_{e}(f,f) $ is called an inelastic Boltzmann collision operator. 
	
The weak formulation of the equation can then be defined as 
	\begin{equation}\label{weak}
		\begin{split}
			&\int_{\mathbb{R}^3 }f(t,v)\psi(v) \,\rd v-\int_{\mathbb{R}^3 }f(0,v)\psi(v) \,\rd v= \int_0^t\int_{\bR^{3}} Q_{e}(g,f)(\tau,v)\psi(v) \,\rd v \,\rd \tau\\
			&= \frac{1}{2}\int_0^t\int_{\bR^{3}} \int_{\bR^{3}} \int_{\Sd^{2}} B(|v-v_*|,\sigma) g(\tau,\vs) f(\tau,v) \\&\qquad\qquad\qquad\qquad\qquad\qquad\times \left[ \psi(v') + \psi(\vs') - \psi(v) - \psi(\vs) \right] \,\rd \sigma \,\rd \vs \,\rd v \,\rd \tau,
		\end{split}
	\end{equation}
	for any test function $ \psi(v) \in C^2_b(\mathbb{R}^3)$, which is continuous and bounded up to its second-order derivative. The reason why the weak form \eqref{weak} is usually preferred in the study of the inelastic equation is that the post-collisional mechanism of the inelastic collision has been merely manifested in the test function $ \psi(v') $ and $ \psi(\vs') $. Although the restitution coefficient $ e $ might depend on the relative velocity of the colliding particles \cite{CHMR2020}, we will consider coefficient $ e $ as a constant in this paper. 
	
	Lastly, we also introduce the conservation of mass and momentum as well as the dissipation of energy during each inelastic collision, which are the direct consequences of \eqref{ve}:
	$$v'+v'_*=v+v_*,$$and \begin{equation}
	    \label{energyloss} |v|^2+|v_*|^2-|v'|^2+|v'_*|^2  = -\frac{1-e^2}{2}\frac{1-(\hat{v}_- \cdot \sigma)}{2}|v-v_*|^2<0,
	\end{equation}
 where we define  $\hat{v}_- =  \frac{v-v_{*}}{|v-v_{*}|}.$
        This further implies that the inelastic collision operator $Q_e$ formally satisfies
	\begin{equation}\label{conservQe}
        \int_{\mathbb{R}^3}
        \begin{pmatrix}	    
        1\\v
        \end{pmatrix} 
        Q_e (f,f)(v) \,\rd v = 0, \text{ but }
	\end{equation}
	\begin{equation}\label{disspationQe}
        \int_{\mathbb{R}^3} |v|^2 Q_e (f,f)(v) \,\rd v \le 0.
	\end{equation}
	
	\subsection{Hard potential collision kernel without angular cutoff}
	
	The Boltzmann collision operator \eqref{weak} contains a collision kernel $ B(|v-v_*|,\sigma) $ which further describes the types of the interactions of colliding particles. It is defined for measuring the intensity of the collision between interacting particles, which depends not only on the deviation angle $ \theta $, but also on the relative velocity $ |v-v_{*}| $. In general, the collision kernel cannot be explicitly defined, except under some specific circumstances such as the inverse-power-law interaction potential, where $ B(|v-v_*|,\sigma) $ is taken in the product form; i.e.,
	\begin{equation}\label{Bb}
		B\left(|v-v_{*}|, \sigma\right) = b\left( \cos\theta \right) \Phi\left(|v-v_{*}|\right) \ \text{with}\  \cos\theta = \hat{v}_- \cdot \sigma.
	\end{equation}
    Thanks to the standard symmetrization, we can restrict the range of $ \theta $ into $ [0,\pi/2] $ without loss of generality in the sense that
	\begin{equation}\notag
		\bar{b}(\cos\theta) = \left[ b(\cos\theta) + b(\cos\left(\pi-\theta\right))\right]\mathbf{1}_{0\leq \theta \leq \pi/2},
	\end{equation}
    which involves an asymptotic singularity in $ b(\cos\theta) $ as the deviation angle $ \theta $ approaches to zero:
	\begin{equation}\label{noncutoffnu}
		\sin \theta b(\cos\theta) \big|_{\theta\rightarrow 0^{+}} \sim K\theta^{-1-2s}, \quad \text{for}\ 0<s<1 \ \text{and} \  K >0.
	\end{equation}
    In order to overcome the difficulties induced from the singular behavior, the so-called Grad's cutoff assumption has been introduced cutting off a small angle nearby the singularity such that 
	\begin{equation}\label{cutoff1}
		\int_{\Sd^{2}} b_n\left(\hat{v}_- \cdot \sigma\right) \,\rd \sigma < \infty.
	\end{equation}
    Here, the angular kernel $b_n$ can be general enough as long as it is integrable in $\sigma$ and $b_n\to b$ as $n\to \infty$. One of the possible examples is the one that we provide in \eqref{bn} for the existence theory.

    For the hard potential collisions without cutoff, we consider $ B(|v-v_*|,\sigma) $ in the form of \eqref{Bb} with the following assumptions throughout the rest of the paper:
	\begin{itemize}
		\item The angular kernel $  b(\cos\theta) $ is not locally integrable but satisfies
		\begin{equation}\label{noncutoffb}
			\quad\exists\text{ some } \alpha_{0} \in (0,2]  \text{ such that } \int_{0}^{\frac{\pi}{2}} \sin^{\alpha_{0}}\left(\frac{\theta}{2}\right) b(\cos\theta) \sin\theta \,\rd\theta < \infty.
		\end{equation}
		
		\item The kinetic kernel $ \Phi(|v-v_{*}|) $ is in a power form and satisfies the \textit{hard potential} assumption that
		\begin{equation}\label{hard}
			\Phi(|v-v_{*}|) = |v-v_{*}|^{\gamma} \quad \text{with} \quad 0<\gamma\le 2. \ 
		\end{equation} 
	\end{itemize}
 
	\begin{remark}
        Several previous results in the case without angular cutoff such as {\rm\cite{AMUXY2010regularizing,GS2011,MUXY2009}} refer the \textit{hard potential} model to the case with an assumption that $\gamma+2s\ge0.$ In this paper, we only consider the case when $0<\gamma\le 2$ in which one can expect the moments-creation property as well. For the remaining case of $-2s\le \gamma\le 0$, we refer to {\rm\cite{Qi21_soft,Qi22}}.
	\end{remark}

	\subsection{Statement of our main results}
	\label{sub:main}
	
    Throughout this paper, we shall use the \textit{Japanese bracket} notation $ \left\langle \cdot \right\rangle = \sqrt{1+|\cdot|^{2}} $, and $ A \lesssim B \ (\text{or }B \lesssim A)$ means that there is a generic constant $ C>0 $ such that $ A \leq C B \ (\text{or }B \leq CA,\text{ resp.}) $.
	
    We would first like to introduce the global existence of a measure-valued solution to the inelastic Boltzmann equation without angular cutoff. We define the space $ P_{\alpha}(\mathbb{R}^{3}) $ as the set of probability measures on $ \mathbb{R}^{3} $ with finite moments up to the order $ \alpha $; i.e., if $P_{0}(\mathbb{R}^{3})$ is denoted as the set of all positive probability measures on $\mathbb{R}^{3}$,
		\begin{equation}\label{Palpha}
			\begin{split}
				P_{\alpha}(\mathbb{R}^{3}) &\eqdef \bigg\{ F(v) \in P_{0}(\mathbb{R}^{3}) \bigg|  \int_{\mathbb{R}^{3}} \,\rd F(v) =1, \ \int_{\mathbb{R}^{3}} |v|^{\alpha} \,\rd F(v) < \infty,\\
				&\text{and if}\ \alpha > 1, \ F\text{ further satisfies } \int_{\mathbb{R}^{3}} v_{j} \,\rd F(v) = 0, \ j = 1,2,3 \bigg\}.
			\end{split}
		\end{equation}
    Then we establish the following existence theorem under a general class of initial conditions without the smallness requirement nor an entropy bound.
	\begin{theorem}[Global existence]\label{main1}
	   Assume that $ e\in(0,1] $ and the collision kernel $ B\left(|v-v_{*}|, \sigma\right) $ satisfies the non-cutoff hard potential assumptions \eqref{noncutoffb} and \eqref{hard} with $ 0<\gamma\le 2 $. 
		For any initial datum $ F_{0}(v) \in P_{2+\kappa}\left(\mathbb{R}^{3}\right) $ with $ \kappa > 0 $, then there exists a measure-valued solution $ F_{t}(v) \in C\left( \left[0,\infty\right), P_{2+\kappa}\left( \mathbb{R}^{3} \right) \right) $ to the Cauchy problem \eqref{IB}-\eqref{F0}.
	\end{theorem}
 
    \begin{remark}
    (i) The definition of the measure-valued solution follows from {\rm{\cite[Definition 1.1]{morimoto2016measure}}}.\\
    (ii) The uniqueness of a measure-valued solution for the inelastic case without angular cutoff has not been known so far in general, to the best of the authors' knowledge. In the elastic case for non-cutoff hard potentials, one can obtain the uniqueness results by assuming the initial boundedness of either the exponential moment $$\int_{\mathbb{R}^d}e^{\varepsilon_0|v|^\gamma} f_0(v) \,\rd v < +\infty,$$ as in {\rm{\cite[Corollary 2.3]{FM2009moderate}}} or a more regular norm 
    $$\|f_0\|_{W^{1,1}_q}<+\infty,\ \text{for} \ q \ge 2,$$ 
    as in {\rm{\cite[Theorem 2]{DM2009stability}}}.  	
    \end{remark}

    In addition, we also present the moments-creation property for any measure-valued solution $ F_{t} $ with a suitable dissipation condition.
	\begin{theorem}[Creation of moments]\label{main2}
    Assume that $ e \in (0,1] $, $ \kappa > 0 $, and the collision kernel $ B\left(|v-v_{*}|, \sigma\right)  $ satisfies the non-cutoff hard potential assumptions \eqref{noncutoffb} and \eqref{hard} with $ 0<\gamma\le 2 $. For any measure-valued solution $ F_{t}(v) \in C\left(  [0,\infty), P_{2+\kappa}\left(\mathbb{R}^{3}\right) \right) $ to problem \eqref{IB}-\eqref{F0} with the initial datum $ F_{0}(v) \in P_{2+\kappa}(\mathbb{R}^{3}) $, the following \textit{uniform-in-time} moments creation property holds: for any $l\in\mathbb{R}$ and $ t_{0} \in \left( 0,\infty \right) $, we have
		\begin{equation}\label{momentcon}
			\sup\limits_{t \geq t_{0}} \int_{\bR^{3}} \left\langle v \right\rangle^{l} \,\rd F_{t}(v) < \infty.
		\end{equation}   
	\end{theorem}
 
	\begin{remark}
        Here we would like to emphasize that the proofs of the global well-posedness and the moments-creation property heavily depend on the establishment of a refined version of the Povzner inequality for inelastic collision under both cutoff and non-cutoff assumption.  For instance, the establishment of an improved Povzner inequality relieves the condition of moments to guarantee the moments-gain property in Theorem \ref{main2} above. One of the key ideas for the establishment of a refined inequality is to work on the center-of-momentum coordinate system \eqref{com} instead of working with the standard inelastic post-collisional representations \eqref{ve}. Our work under the presence of an angular singularity is motivated by the work in the elastic case {\rm\cite{MW1999,morimoto2016measure}}. Since the inelastic case involves additional parameters that directly affect the decomposition of collisional scattering angle, this requires a more technical (but careful) analysis of the computations of the Jacobian determinants for the varied changes of variables during the journey of the proof of the refined Povzner inequality {\rm(Proposition \ref{Povzner})}. 
	\end{remark}
	
	\subsection{Outline of the paper}
	\label{sub:plan}
	
	The rest of the paper is organized as follows. In Section~\ref{sec:povzner}, we will first establish a refined Povzner-type estimate for the inelastic Boltzmann equation under both cutoff and non-cutoff assumptions, which works as a powerful tool in the following sections. Using the Povzner estimate, we establish the existence theory of the measure-valued solution to the inelastic Boltzmann equation with hard potential in Section~\ref{sec:existence}. Then the moments-creation property of the measure-valued solution will further be provided in Section~\ref{sec:moment}.

	\section{New non-cutoff Povzner estimates for the inelastic hard potentials}
	\label{sec:povzner}

	One of the main properties that arise in the hard potential case is the moments-creation property. 
	For the proof of global well-posedness and the moments-creation property under the non-cutoff situation, the establishment of an improved Povzner-type estimate is crucial. The original form of the Povzner-type estimate is the useful inequality for $\triangle \psi = \psi(v'^2) + \psi(v'^2_*) - \psi(v^2) - \psi(v_*^2)$ with $\psi(x) = x^p$, proposed by Povzner \cite{Povzner}, which has been improved by Elmroth \cite{Elmroth83} and Wennberg \cite{Wennberg97} to obtain the moments estimates for the elastic Boltzmann equation under the cutoff assumption. Still, for the  case of elastic collision, Bobylev \cite{Bobylev97_moment} and Mischler-Wennberg \cite{MW1999} further improve this type estimate towards the integral of $\triangle \psi$ in the average sense, instead of $\triangle \psi$ itself. In terms of the inelastic Boltzmann equation, the Povzner-type estimate is derived by Bobylev-Gamba-Panferov \cite{BGP04} and extended for more general $\psi$ by Gamba-Panferov-Villani \cite{GambaDiffusively} under the cutoff assumption.
 
    Therefore, all the existing Povzner-type estimates are unfortunately not sufficient enough to prove the well-posedness of the inelastic Boltzmann equation in the non-cutoff case and the moments-creation property to the best of the authors' knowledge. Therefore, we establish in this section a non-cutoff version of the Povzner inequality for the inelastic Boltzmann equation for the proof of the moments-creation property.

	\begin{proposition}\label{Povzner}
		Assume that $ e\in(0,1] $ and the angular cross-section $b$ satisfies \eqref{noncutoffnu} and \eqref{noncutoffb}. Let  $ \psi $ denote a convex function \begin{equation}\label{psionetwo} \psi_{1}(x) \eqdef x^{1+\frac{\kappa}{2}}\ \text{  or  }\ \psi_{2}(x) \eqdef  (1+x)^{1+\frac{\kappa}{2}}-1,\end{equation} with $ \kappa > 0 $. Let  $ K^{e}(v,v_{*}) $ be defined for every $ \kappa > 0 $ as follows:
		\begin{equation}\label{Ke1}
			\begin{split}
				K^{e}(v,v_{*}) = \int_{\Sd^{2}} b(\hat{v}_- \cdot\sigma) \left[ \psi(|v'|^{2}) + \psi(|\vs'|^{2}) - \psi(|v|^{2}) - \psi(|\vs|^{2}) \right] \,\rd\sigma.
			\end{split}
		\end{equation}
		Then, we can represent $K^e$ as $ K^{e}(v,v_{*}) = H(v,v_{*}) + G(v,v_{*}) $ where $H$ and $G$ further satisfy the following bounds:
		\begin{equation}\label{Hbound}
			 H(v,v_{*}) \leq -C_1(|v|^{2+\kappa}|v_*|^2+|v_*|^{2+\kappa}|v|^2)  +C_2 (\langle v\rangle ^{1+\kappa}|v_*|+\langle v_*\rangle^{1+\kappa}|v|),
		\end{equation}
		and
		\begin{equation*}
			G(v,v_{*})\le \left\{ 
			\begin{aligned}
				& \ C_3 |v|^{2} |v_{*} |^{2}, \ \text{if}\ \kappa < 2, \\
				& \ C_4 \left( |v_{*}|^{2} \left\langle v\right\rangle^{\kappa} + |v|^{2} \left\langle v_{*}\right\rangle^{\kappa} \right), \ \text{if}\ \kappa \geq 2,
			\end{aligned}
			\right.
		\end{equation*}
		where $C_1,\ C_2,\ C_3$, and $C_4$ are constants that depend on $\lambda$, $ e $ and $\kappa$.
	\end{proposition}
 \begin{remark}
     In Proposition \ref{Povzner}, if $\psi$ is chosen to be $\psi_2$, then we have a better bound in \eqref{Hbound} as follows:
     \begin{equation}\label{Hboundbetter}
			 H(v,v_{*}) \leq -C_1(\langle v\rangle^{2+\kappa}|v_*|^2+\langle v_*\rangle ^{2+\kappa}|v|^2)  +C_2 (\langle v\rangle ^{1+\kappa}|v_*|+\langle v_*\rangle^{1+\kappa}|v|).
		\end{equation}
 \end{remark}
    Before we move onto the proof, we first introduce several different coordinates that we use. Our key steps are on the several series of changes of angular variables. The first one is that we decompose the angular variable into the standard polar coordinates $ \sigma \longmapsto (\theta, \phi)$ with a specific choice of the $z$-axis motivated by \cite{MW1999,morimoto2016measure}, where the authors take advantage of the Povzner estimates in the elastic case. Then the key step is that we further proceed to define our own special angles $\chi,\mu$ and take  $(\theta, \phi) \longmapsto (\chi,\mu) $ to further improve the estimates even in the non-cutoff inelastic situation.

    To this end, we first define and write the cutoff version $ K_{n}^{e}(v,v_{*}) $ of the operator $K^e$ as
		\begin{equation}\label{Ke}
			\begin{split}
				K_{n}^{e}(v,v_{*}) = \int_{\Sd^{2}} b_{n}(\hat{v}_-\cdot\sigma) \left[ \psi(|v'|^{2}) + \psi(|\vs'|^{2}) - \psi(|v|^{2}) - \psi(|\vs|^{2}) \right] \,\rd\sigma,
			\end{split}
		\end{equation}where $b_n$ is a mollified angular kernel with Grad's cutoff that satisfies \eqref{cutoff1}. This includes the specific one \eqref{bn} that we use.
    In terms of distinguishing inelastic collisions, it would be more advantageous to parametrize the post-collisional velocities $ v' $ and $ v'_{*} $ in the \textit{center-of-momentum} coordinate system as illustrated in Figure \ref{fig1}, where the similar coordinate system used to be applied in~\cite{BGP04,GambaDiffusively} as well.
	

	\begin{figure}
	\centering
    \begin{tikzpicture}
        [scale=4,
		>=stealth,
		point/.style = {draw, circle,  fill = black, inner sep = 1pt},
		dot/.style   = {draw, circle,  fill = black, inner sep = .2pt},
		]
    \input{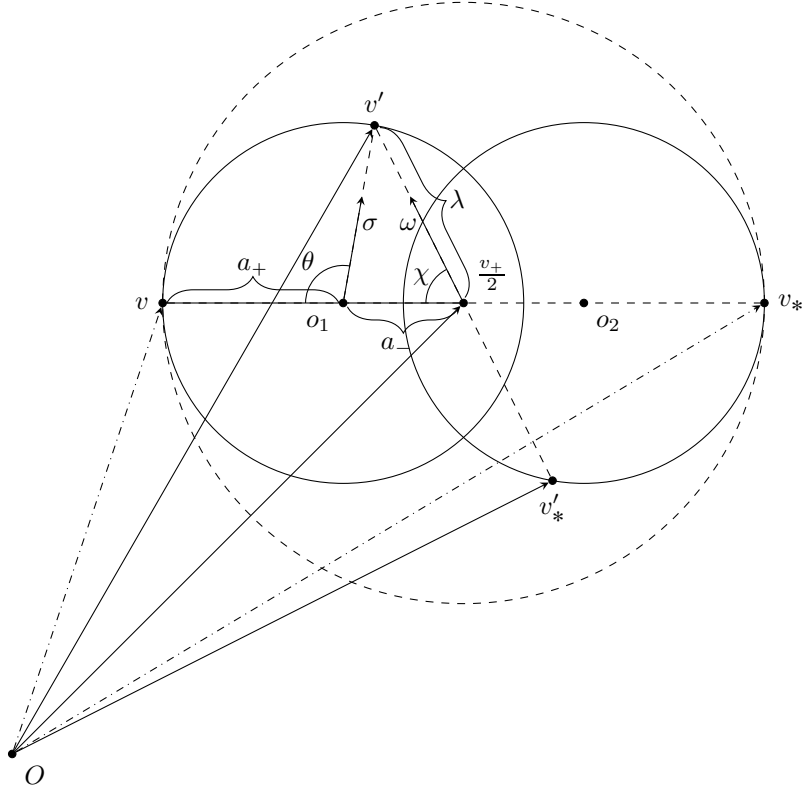}
    \end{tikzpicture}
    \caption{Illustration of the \textit{center-of-momentum} coordinate system in an inelastic collisional process.} 
    \label{fig1}
    \end{figure}
    
	
	We begin with setting
		\begin{equation}\label{com}
			v' = \frac{v_{+} + \lambda|v_{-}|\omega}{2} \quad \text{and}\quad v'_{*}= \frac{v_{+} - \lambda|v_{-}|\omega}{2},
		\end{equation}
		where $ v_{+} = v + v_{*} $, $ v_{-} = v - v_{*} $, and $ \omega \in \mathbb{S}^{2} $ is a unit vector. Then we observe
		\begin{equation}\label{LO}
			\lambda \omega = a_{+} \sigma + a_{-}\hat{v}_{-},
		\end{equation}
		with $ a_{+} = \frac{1+e}{2} $, $ a_{-} = \frac{1-e}{2} $  and $ \hat{v}_{-} = \frac{v_{-}}{|v_{-}|} $. Therefore, one can see from Figure \ref{fig1} that
		\begin{equation}\notag
			\lambda = \lambda(\cos\chi) = a_{-}\cos\chi + \sqrt{a_{-}^{2}(\cos^{2}\chi - 1) + a_{+}^{2}},
		\end{equation}
		where $ \chi $ is the angle between $ v_{-} $ and $ \omega $. Then, by a direct calculation or \cite[pp. 514]{GambaDiffusively} we note that
		\begin{equation*}
			0 < e \leq \lambda(\cos\chi) \leq 1,
		\end{equation*}
		for all $ \chi $. With this parametrization, we can represent the size of the post-collisional velocities as
		\begin{equation}\label{v'2}
			\begin{split}
				|v'|^{2} &= \frac{|v_{+}|^{2}+ \lambda^{2}|v_{-}|^{2} + 2\lambda|v_{+}||v_{-}|\cos\mu}{4}\\
				&= Y(\chi) + Z(\chi)\cos\mu,
			\end{split}
		\end{equation}
        and
		\begin{equation}\label{vs'2}
			\begin{split}
				|v'_{*}|^{2} &= \frac{|v_{+}|^{2}+ \lambda^{2}|v_{-}|^{2} -2\lambda|v_{+}||v_{-}|\cos\mu}{4}\\
				&= Y(\chi) - Z(\chi)\cos\mu,
			\end{split}
		\end{equation}
		where 
            \begin{equation}\label{yzdef} 
            Y(\chi) = \frac{|v_{+}|^{2} + \lambda^{2}|v_{-}|^{2}}{4}, \quad   Z(\chi) = \frac{\lambda|v_{+}||v_{-}|}{2}. 
            \end{equation} 
        Here $ \mu $ is the angle between the vector $ v_{+} $ and $ \omega $. Moreover, by taking an inner-product with $ \hat{v}_{-} $ on both sides of relation~\eqref{LO}, we can also derive that
        \begin{equation}\label{lamda}
            \lambda \cos\chi = a_{+} \cos\theta + a_{-}.
        \end{equation}
        Then, this provides
		\begin{equation}\label{Lo1}
			\begin{split}&\cos\theta = \frac{\lambda\cos\chi - a_{-}}{a_{+}} = \frac{\left[a_{-}\cos\chi + \sqrt{a_{-}^{2}(\cos^{2}\chi - 1) + a_{+}^{2}} \right]\cos\chi - a_{-}}{a_{+}},
			\end{split}
		\end{equation}
	where $ \theta $ is the deviation angle between $ \sigma $ and $ \hat{v}_- $, manifested in the standard decomposition of $ \sigma \in \Sd^{2} $.
		
    We would like to mention that the estimate with respect to $ \sigma $ is now transferred to the estimates of $ \theta $ and $ \phi $, and the singularity of $ \theta $ can now be canceled by the ``good'' terms provided from the integration by part in $ \phi $.
	Consequently, by noticing that $ |v'|^{2} $ and $ |v'_{*}|^{2} $ have been represented in the angle $ \chi $ and $ \mu $ instead of $ \theta $ and $ \phi $, we can have a refined version of the inelastic Povzner inequality which follows from the Jacobian determinant of the transformation $ \sigma \longmapsto (\theta, \phi) \longmapsto (\chi,\mu) $. This Jacobian determinant will be computed in the following proof. The proof also requires the use of the convexity of specifically chosen functions $\psi$, as well as the use of the monotonicity of $\psi$ for $x>0$. These conditions will be used to provide additional room to cancel the angular singularity.

	\begin{proof}[Proof of Proposition \ref{Povzner}]
		As the first step of the standard polar coordinate change $ \sigma \longmapsto (\theta, \phi) $ is now clear, our focus will now be on the second step $(\theta, \phi) \longmapsto (\chi,\mu) $.
        We first attempt to find the differential relation between $ \cos\theta = A $ and $ \cos\chi = B$ such that, by \eqref{Lo1},
		\begin{equation}\notag
			\begin{split}
				\frac{\rd A}{\rd B} & =  \frac{1}{a_{+}} \left[ \frac{a_{-}^{2}B^{2}}{\sqrt{a_{-}^{2}(B^{2}-1)+a_{+}^{2}}}+ 2a_{-}B  + \sqrt{a_{-}^{2}(B^{2}-1)+a_{+}^{2}} \right]\\
				& = \frac{\left[ a_{-}B + \sqrt{a_{-}^{2}(B^{2}-1)+a_{+}^{2}}\right]^{2}}{a_{+}\sqrt{a_{-}^{2}(B^{2}-1)+a_{+}^{2}}}.
			\end{split}
		\end{equation}
		In the original representation in the spherical coordinates 
        $$ \sigma = (\cos\theta, \ \sin\theta \cos\phi, \  \sin\theta \sin\phi) ,$$ 
            the singularity of the collision kernel $ b(\cos\theta) $ appears when $ \theta \rightarrow 0 $, i.e., $ A= \cos\theta \rightarrow 1 $. By  \eqref{Lo1}, this corresponds to $ B = \cos\chi \rightarrow 1$ as well. Furthermore, when $ A = \cos\theta \rightarrow 0 $, it again follows from the direct calculation via \eqref{Lo1} that,
		\begin{equation}\label{Lo2}
			\begin{split}
				&\ 0 = A = \frac{\left[a_{-}B + \sqrt{a_{-}^{2}(B^{2} - 1) + a_{+}^{2}} \right]B - a_{-}}{a_{+}},  \\
				\Longrightarrow &\  \left[a_{-}B + \sqrt{a_{-}^{2}(B^{2} - 1) + a_{+}^{2}} \right]B = a_{-},\\
				\Longrightarrow&\  B\sqrt{a_{-}^{2}(B^{2} - 1) + a_{+}^{2}} = a_{-} - a_{-}B^2,\\
				\Longrightarrow&\  B^2[a_{-}^{2}(B^{2} - 1) + a_{+}^{2}] = a_{-}^{2}\left(1- 2B^2+B^{4}\right),\\
				\Longrightarrow&\  B^{2}(a_{+}^{2} + a_{-}^{2}) = a_{-}^{2},\\
				\Longrightarrow&\  B = \frac{a_{-}}{\sqrt{a_{+}^{2} + a_{-}^{2}}}.
			\end{split}
		\end{equation}
		Thus, the limit $A\to 0$ corresponds to the limit $ B = \cos\chi \rightarrow \frac{a_{-}}{\sqrt{a_{+}^{2} + a_{-}^{2}}} $.
	Hence, the change of variables $  (\theta,\phi) \longmapsto (\chi,\mu) $ can be re-written as $  (A,\phi) \longmapsto (B,\mu)  $, and we further note that the Jacobian determinant $  \left| \frac{\partial (A,\phi)}{\partial(B,\mu)} \right| $ is
		\begin{equation}\label{J1}
			\left| \frac{\partial (A,\phi)}{\partial(B,\mu)} \right| \eqdef\left|
			\begin{array}{cc}
				\frac{\partial A}{\partial B} & \frac{\partial A}{\partial \mu} =0 \\
				\frac{\partial \phi}{\partial B} & \frac{\partial \phi}{\partial \mu} \\
			\end{array}
			\right|.
		\end{equation}
    Thus, in order to compute the Jacobian determinant $ \left| \frac{\partial (A,\phi)}{\partial(B,\mu)} \right| $, we need to calculate $ \Big| \frac{\partial \phi}{\partial \mu} \Big| $. To express the relationship between $\mu$ and $\phi$, we devise a different coordinate system as follows.	In the coordinate system $ (\hat{v}_{-}, \hat{j}, \hat{h}) $, where 
    \begin{equation*}
     \hat{j} = \frac{v \times v_*}{|v \times v_*|}, \quad  \hat{h} = \hat{v}_{-} \times \hat{j} = \frac{((v-v_*)\cdot v)v_* - ((v-v_*)\cdot v_*)v}{|v-v_*||v \times v_*|},
     \end{equation*}
    we have, considering the fact that $ \hat{v}_{+} $ lies on the plane spanned by $ (\hat{v}_{-}, \hat{h}) $,
		\begin{equation*}
			\left.
			\begin{aligned}
				\hat{v}_{+} &= (\hat{v}_{+}\cdot \hat{v}_{-}) \hat{v}_{-} + (\hat{v}_{+}\cdot \hat{h}) \hat{h}\\
				\cos\mu &= \omega \cdot \hat{v}_{+}\\
				\omega &= \frac{a_{+}\sigma + a_{-}\hat{v}_{-}}{\lambda}
			\end{aligned}
			\right\} \Longrightarrow 
			\begin{aligned}
				\cos\mu = \cos\beta \frac{a_{+}\cos\theta + a_{-}}{\lambda} + \sin\beta \frac{a_{+}\sigma\cdot\hat{h}}{\lambda},
			\end{aligned} 
		\end{equation*}
        where $ \beta $ is the angle between $ \hat{v}_{+} $ and $ \hat{v}_{-} $. By further noticing \eqref{Lo1}, we obtain the expression of $ \cos\mu $ that
		\begin{equation}\label{cosmu}
			\begin{split}
				\cos\mu &= \cos\beta\cos\chi + \sin\beta\frac{a_{+}\sin\theta\sin\phi}{\lambda}\\
				&=\cos\chi \left( \cos\beta + \frac{a_{+}\sin\beta \sin\theta }{a_{+}\cos\theta + a_{-}} \sin\phi \right),
			\end{split}
		\end{equation}
		to which we take the derivative with respect to $ \phi $ such that
        \begin{equation}\label{sinmu}
            -\sin\mu \frac{\partial \mu}{\partial \phi} = \frac{a_{+}\sin\beta \sin\theta \cos\chi}{a_{+}\cos\theta + a_{-}} \cos\phi.
        \end{equation}
		Therefore, we have
		\begin{equation}\label{partial1}
			\begin{split}
				\frac{\partial \phi}{\partial \mu} &= - \frac{(a_{+}\cos\theta + a_{-})\sin\mu}{a_{+}\sin\beta \sin\theta \cos\chi \cos\phi}\\
				&= -\frac{\lambda\sin\mu}{a_{+}\sin\beta\sin\theta\cos\phi} \\
				&= -\frac{\sin\mu \tan\phi}{\cos\mu - \cos\beta\cos\chi},
			\end{split}
		\end{equation} by \eqref{Lo1} and \eqref{cosmu}. 
		To obtain $ \tan \phi $, we recall \eqref{cosmu} and find that
		\begin{equation*}
			\begin{split}
				\sin\phi &=\frac{\lambda(\cos\mu-\cos\beta\cos\chi)}{a_{+}\sin\beta\sin\theta}\\
				&=\frac{\lambda(\cos\mu-\cos\beta\cos\chi)}{\sin\beta\sqrt{a_{+}^{2}-(\lambda\cos\chi-a_{-})^{2}}},
			\end{split}
		\end{equation*}
		from which, we derive that
		\begin{equation}\label{sinphi}
			\tan\phi = \pm\frac{\sin\phi}{\sqrt{1-\sin^{2}\phi}} = \pm\frac{\lambda(\cos\mu-\cos\beta\cos\chi)}{\sqrt{\sin^{2}\beta[a_{+}^{2}-(\lambda\cos\chi-a_{-})^{2}]-\lambda^{2}(\cos\mu-\cos\beta\cos\chi)^{2}}},
		\end{equation} where we use $+$ if $\cos\phi> 0$ and use $-$ if $\cos\phi<0$.
		Combing \eqref{partial1} and \eqref{sinphi}, we finally obtain
		\begin{equation}\label{phimu}
			\frac{\partial \phi}{\partial \mu} = \mp \frac{\lambda \sin\mu}{\sqrt{\sin^{2}\beta[a_{+}^{2}-(\lambda\cos\chi-a_{-})^{2}]-\lambda^{2}(\cos\mu-\cos\beta\cos\chi)^{2}}}.
		\end{equation}
		Then, in order to have a better-reduced representation, we apply another change of variable
		\begin{equation}\label{etamu}
			\eta = \cos\mu - \cos\beta \cos\chi,
		\end{equation}
		such that $ \rd \eta = -\sin\mu\rd \mu $. Then from \eqref{cosmu}, we have
		\begin{equation}\notag
			\sin\phi = \frac{\lambda}{\sin\beta\sqrt{a_{+}^{2}-(\lambda\cos\chi-a_{-})^{2}}} \eta,
		\end{equation}
		which implies the following correspondence:
		\begin{equation}\notag
			\left\{ 
			\begin{aligned}
				\phi &=\frac{\pi}{2} &\Longleftrightarrow& \ \eta = \eta_{0}, \\
				\phi &= \pi &\Longleftrightarrow&\ \eta = 0,\\
				\phi &=\frac{3\pi}{2} &\Longleftrightarrow& \ \eta = -\eta_{0}, \\
			\end{aligned}
			\right.
		\end{equation}
		where 
		\begin{equation}\label{eta0}
			\eta_{0} = \frac{\sin\beta\sqrt{a_{+}^{2}-(\lambda\cos\chi-a_{-})^{2}}}{\lambda}.
		\end{equation}
	   Then \eqref{phimu} is reduced into a much simpler form of
		\begin{equation}\label{phimu1}
			\frac{\partial \phi}{\partial \mu} = \mp \frac{\sin\mu}{\sqrt{\eta_{0}^{2} -\eta^{2}}}.
		\end{equation}
		
		Finally, we have  
		\begin{multline*}
			\int_{\Sd^{2}} b_{n}(\hat{v}_{-}\cdot\sigma) \,\rd \sigma  = 2\int_{\frac{\pi}{2}}^{\frac{3\pi}{2}} \int_{0}^{1} b_{n}(A)  \,\rd A \,\rd \phi \\= 2 \int_{\frac{a_{-}}{\sqrt{a_{+}^{2} + a_{-}^{2}}}}^{1} b_{n}(B) \Big| \frac{\rd A}{\rd B} \Big| \int_{-\eta_{0}}^{\eta_{0}}  \Big| \frac{\rd \phi}{\rd \mu} \Big| \Big| \frac{\rd \mu}{\rd \eta} \Big|  \,\rd \eta \,\rd B.
		\end{multline*}
	   We put this into the estimate for $ K_{n}(v,v_{*}) $ and obtain
		\begin{equation}\label{Kn}
			\begin{split}
				K_{n}(v,v_{*})&=\int_{\Sd^{2}} b_{n} (\hat{v}_{-}\cdot\sigma) \left[ \psi(|v'|^{2}) + \psi(|v_{*}'|^{2}) -\psi(|v|^{2}) - \psi(|v_{*}|^{2}) \right] \,\rd \sigma \\[2pt]
				&= \int_{0}^{2\pi} \int_{0}^{\frac{\pi}{2}} b_{n}(\cos\theta) \sin\theta  \left[ \psi(|v'|^{2}) + \psi(|v_{*}'|^{2}) -\psi(|v|^{2}) - \psi(|v_{*}|^{2}) \right] \,\rd \theta \,\rd \phi\\[2pt]
				&= 2 \int_{0}^{1} b_{n}(A) \int_{\frac{\pi}{2}}^{\frac{3\pi}{2}} \left[ \psi(|v'|^{2}) + \psi(|v_{*}'|^{2}) -\psi(|v|^{2}) - \psi(|v_{*}|^{2}) \right] \,\rd \phi \,\rd A \\[2pt]
				&= 2 \int_{\frac{a_{-}}{\sqrt{a_{+}^{2} + a_{-}^{2}}}}^{1} b_{n}(B) \\&\times  \Big| \frac{\rd A}{\rd B} \Big| \int_{-\eta_{0}}^{\eta_{0}} \left[ \psi(|v'|^{2}) + \psi(|v_{*}'|^{2}) -\psi(|v|^{2}) - \psi(|v_{*}|^{2}) \right] \Big| \frac{\rd \phi}{\rd \mu} \Big| \Big| \frac{\rd \mu}{\rd \eta} \Big|  \,\rd \eta \,\rd B,
			\end{split}
		\end{equation} 
  where the symmetry of $\phi$ in the third equality can be noticed from the representation of post-collisional velocities \eqref{v'2}-\eqref{vs'2} and the relation between $\mu$ and $\phi$ in \eqref{cosmu}.

  Now our job is to compute each integral $\int_{-\eta_{0}}^{\eta_{0}} \psi(|u|^{2}) \Big| \frac{\rd \phi}{\rd \mu} \Big| \Big| \frac{\rd \mu}{\rd \eta} \Big|  \,\rd \eta$ for $u=v'$ and $u=v'_*$.
		Noting \eqref{etamu} and \eqref{phimu1}, we first observe that
		\begin{equation}\notag
			\begin{split}
				&\int_{-\eta_{0}}^{\eta_{0}} \psi(|v'|^{2}) \Big| \frac{\rd \phi}{\rd \mu} \Big| \Big| \frac{\rd \mu}{\rd \eta} \Big|  \,\rd \eta \\
				&= \left(\int_{0}^{\eta_{0}} + \int^{0}_{-\eta_{0}} \right) \psi(Y(\chi) + Z(\chi)(\cos\beta B+\eta)) \frac{1}{\sqrt{\eta_{0}^{2} - \eta^{2}}} \,\rd \eta\\
				&= \int_{0}^{\eta_{0}} \left[ \psi(Y(\chi) + Z(\chi)(\cos\beta B+\eta)) + \psi(Y(\chi) + Z(\chi)(\cos\beta B-\eta)) \right] \frac{1}{\sqrt{\eta_{0}^{2} - \eta^{2}}} \,\rd \eta\\
				&= \int_{0}^{\eta_{0}} \bigg[ \psi(Y(\chi) + Z(\chi)(\cos\beta B+\eta)) + \psi(Y(\chi) + Z(\chi)(\cos\beta B-\eta)) \\
				&-2\psi(Y(\chi)+Z(\chi)\cos\beta B) \bigg] \frac{1}{\sqrt{\eta_{0}^{2} - \eta^{2}}} \,\rd \eta +\pi\psi(Y(\chi)+Z(\chi)\cos\beta B),
			\end{split}
		\end{equation}where $Y$ and $Z$ are defined as in \eqref{yzdef}.
		Taking the integration by parts twice and denoting $ Y\eqdef Y(\chi),\  Z\eqdef Z(\chi) $ for the sake of simplicity, we have
		\begin{multline}\label{psiv'}
				 \int_{-\eta_{0}}^{\eta_{0}} \psi(|v'|^{2}) \Big| \frac{\rd \phi}{\rd \mu} \Big| \Big| \frac{\rd \mu}{\rd \eta} \Big|  \,\rd \eta \\
				=  - \int_{\eta=0}^{\eta=\eta_{0}} \left[ \psi(Y + Z(\cos\beta B+\eta)) + \psi(Y + Z(\cos\beta B-\eta)) -2\psi(Y+Z\cos\beta B) \right] \\\times \rd \left[ \cos^{-1}\left(\frac{\eta}{\eta_{0}}\right) \right]
				 +\pi\psi(Y+Z\cos\beta B)\\
				=  \bigg\lbrace -\cos^{-1}\left(\frac{\eta}{\eta_{0}}\right) \bigg[ \psi(Y + Z(\cos\beta B+\eta)) + \psi(Y + Z(\cos\beta B-\eta)) \\-2\psi(Y+Z\cos\beta B) \bigg] \bigg\rbrace \Big|_{\eta=0}^{\eta=\eta_{0}}\\
				 +Z\int_{0}^{\eta_{0}} \cos^{-1}\left(\frac{\eta}{\eta_{0}}\right) \left[ \psi'(Y + Z(\cos\beta B+\eta)) - \psi'(Y + Z(\cos\beta B-\eta)) \right] \,\rd \eta\\
				 +\pi\psi(Y+Z\cos\beta B)\\
				=  Z \int_{\eta=0}^{\eta=\eta_{0}} \left[ \psi'(Y + Z(\cos\beta B+\eta)) - \psi'(Y + Z(\cos\beta B-\eta)) \right] \\\times \rd \left[ \eta\cos^{-1}\left(\frac{\eta}{\eta_{0}}\right) - \sqrt{\eta_{0}^{2}-\eta^{2}} \right]
				 +\pi\psi(Y+Z\cos\beta B)\\
				=  \bigg\lbrace Z \left[ \eta\cos^{-1}\left(\frac{\eta}{\eta_{0}}\right) - \sqrt{\eta_{0}^{2}-\eta^{2}} \right] \\\times \left[ \psi'(Y + Z(\cos\beta B+\eta)) - \psi'(Y + Z(\cos\beta B-\eta)) \right] \bigg\rbrace \Big|_{\eta=0}^{\eta=\eta_{0}}\\
				  - Z^{2}\int_{0}^{\eta_{0}} \left[ \eta\cos^{-1}\left(\frac{\eta}{\eta_{0}}\right) - \sqrt{\eta_{0}^{2}-\eta^{2}} \right]\\\times \left[ \psi''(Y + Z(\cos\beta B+\eta)) + \psi''(Y + Z(\cos\beta B-\eta)) \right] \,\rd \eta \\
				 +\pi\psi(Y+Z\cos\beta B)\\
				=  Z^{2}\int_{0}^{\eta_{0}} \left[ \sqrt{\eta_{0}^{2}-\eta^{2}} -\eta\cos^{-1}\left(\frac{\eta}{\eta_{0}}\right) \right]\\\times  \left[ \psi''(Y + Z(\cos\beta B+\eta)) + \psi''(Y + Z(\cos\beta B-\eta)) \right] \,\rd \eta\\
				 +\pi\psi(Y+Z\cos\beta B).
			\end{multline}
		Similarly, for $ \psi(|v'_{*}|^{2}) $, we have
		\begin{multline}\label{psivstar'}
				\int_{-\eta_{0}}^{\eta_{0}} \psi(|v'_{*}|^{2}) \Big| \frac{\rd \phi}{\rd \mu} \Big| \Big| \frac{\rd \mu}{\rd \eta} \Big|  \,\rd \eta \\
            =\pi\psi(Y-Z\cos\beta B) + Z^{2}\int_{0}^{\eta_{0}} \left[ \sqrt{\eta_{0}^{2}-\eta^{2}} -\eta\cos^{-1}\left(\frac{\eta}{\eta_{0}}\right) \right] \\\times \left[ \psi''(Y - Z(\cos\beta B+\eta)) + \psi''(Y - Z(\cos\beta B-\eta)) \right] \,\rd \eta.
		\end{multline}

		Thanks to the identities \eqref{psiv'} and \eqref{psivstar'} above, $ K_{n}(v,v_{*}) $ of \eqref{Kn} can now be divided into two parts:
		\begin{equation*}
			K_{n}(v,v_{*}) = H_{n}(v,v_{*}) + G_{n}(v,v_{*}),
		\end{equation*}
            where
              \begin{multline*}
		G_{n}(v,v_*) \eqdef 2 \int_{\frac{a_{-}}{\sqrt{a_{+}^{2} + a_{-}^{2}}}}^{1} b_{n}(B)  \Big| \frac{\rd A}{\rd B}  \Big|\\\times 
  Z^{2}\int_{0}^{\eta_{0}} \left[ \sqrt{\eta_{0}^{2}-\eta^{2}} -\eta\cos^{-1}\left(\frac{\eta}{\eta_{0}}\right) \right] \Big[ \psi''(Y + Z(\cos\beta B+\eta)) \\+ \psi''(Y + Z(\cos\beta B-\eta)) + \psi''(Y - Z(\cos\beta B+\eta)) + \psi''(Y - Z(\cos\beta B-\eta)) \Big] \,\rd \eta\,\rd B,
	\end{multline*}and 
            \begin{equation*}
			H_{n}(v,v_{*}) = H_n^1 - H_n^2,
		\end{equation*}
  with
            \begin{equation*}
			H_n^1 \eqdef 2 \pi \int_{\frac{a_{-}}{\sqrt{a_{+}^{2} + a_{-}^{2}}}}^{1} b_{n}(B)    \Big[ \psi(|v'|^2 + |v'_*|^2) - \psi(|v|^2) - \psi(|v_*|^2)   \Big]   \Big| \frac{\rd A}{\rd B} \Big| \,\rd B,
		\end{equation*}
            and
		\begin{multline*}
			H_{n}^{2} \eqdef 2 \pi \int_{\frac{a_{-}}{\sqrt{a_{+}^{2} + a_{-}^{2}}}}^{1} b_{n}(B)   \Big[ \psi(|v'|^2 + |v'_*|^2) \\- \psi(Y+Z\cos\beta B) - \psi(Y-Z\cos\beta B) \Big]
             \Big| \frac{\rd A}{\rd B} \Big|  \,\rd B.
		\end{multline*}

    As for $H_n^1$, by transforming back in the variable of $\theta$ and using the similar argument in the elastic case \cite[pp. 891]{morimoto2016measure}, there exist $[\theta_1, \theta_2] \subset (0,\frac{\pi}{2})$ and $ c_1 > 0 $ independent of $n$ such that $b_n(\cos\theta)\sin\theta \geq c_1$ on $[\theta_1, \theta_2]$, it follows that, there exists a constant $C_{1,1}$ independent of $n$ such that,
    \begin{equation}\label{-Hn1}
    \begin{split}
        -H_{n}^{1} \geq &  \ 2\pi c_1(\theta_2 - \theta_1) \Big[  \psi(|v|^2) + \psi(|v_*|^2) - \psi(|v|^2 + |v_*|^2)\\
        &\qquad \qquad \qquad \qquad\qquad \qquad + \psi(|v|^2 + |v_*|^2) -  \psi(|v'|^2 + |v'_*|^2)  \Big] \\
        \geq & -C_{1,1} \Big[ |v|^2 \psi'(|v_*|^2) + |v_*|^2 \psi'(|v|^2) \Big],
    \end{split}
    \end{equation}
    where, in the last inequality, since both of the functions $\psi_1$ and $\psi_2$ in \eqref{psionetwo} satisfy the conditions (3.4) - (3.7) listed in \cite[pp. 513]{GambaDiffusively} \footnote{For the purpose of completeness, the conditions and lemma will be included in the appendix}, we can apply the inequalities in \cite[Lemma 3.1, (3.8)]{GambaDiffusively} for the term $\psi(|v|^2) + \psi(|v_*|^2) - \psi(|v|^2 + |v_*|^2)$, and the monotonicity of $\psi$ and $|v|^2 + |v_*|^2 \geq |v'|^2 + |v'_*|^2$ are utilized for $\psi(|v|^2 + |v_*|^2) -  \psi(|v'|^2 + |v'_*|^2)$.
    Hence, \eqref{-Hn1} further implies that 
    \begin{equation}\label{Hn1}
    \begin{split}
        H_{n}^{1} \leq&\  C_{1,1} \Big[ |v|^2 \psi'(|v_*|^2) + |v_*|^2 \psi'(|v|^2) \Big] \\
        \leq & \  C_{1,1}\left(1+\frac{\kappa}{2}\right) \Big[ |v|^2 \left\langle v_{*}\right\rangle^{\kappa} + |v_*|^2 \left\langle v\right\rangle^{\kappa} \Big],
    \end{split}
    \end{equation}
    where $\psi'$ can be taken as $\psi'_1$ or $\psi'_2$.

   As for $H_n^2$, by noting $|v'|^2 + |v'_*|^2 = 2Y$ from \eqref{v'2} - \eqref{vs'2} and using \eqref{A3} for the lower bound of the term $\psi(|v'|^2 + |v'_*|^2) - \psi(Y+Z\cos\beta B) - \psi(Y-Z\cos\beta B)$, we have, for some $c_2>0$
    \begin{equation}\notag
    \begin{split}
        &\psi(|v'|^2 + |v'_*|^2) - \psi(Y+Z\cos\beta B) - \psi(Y-Z\cos\beta B) \\ 
        = &\ \psi(2Y) - \psi(Y+Z\cos\beta B) - \psi(Y-Z\cos\beta B) \\
        \geq &\ c_2 (Y+Z\cos\beta B)(Y-Z\cos\beta B) \psi''(2Y)\\
        \geq & \ c_2 (Y^2 - Z^2) \psi''(2Y),
    \end{split}
    \end{equation} because $B=\cos\chi$ and $|\cos\beta B|\le 1.$
    Furthermore, recalling \eqref{yzdef}, we have
    \begin{equation}\label{pp1}
    \begin{split}
        Y^2 - Z^2 =&\ \left(\frac{|v_{+}|^{2} + \lambda^{2}|v_{-}|^{2}}{4}\right)^{2}- \left( \frac{\lambda|v_{+}||v_{-}|}{2}\right)^2\\[3pt]
        = &\ \frac{\left( |v_{+}|^{2} - \lambda^{2}|v_{-}|^{2} \right)^2}{16} \\[3pt]
        = &\ \frac{\left[ (1-\lambda^2)(|v|^{2}+|v_*|^{2}) +2(1+\lambda^2)v\cdot v_* \right]^2}{16} \\
       \geq & \ \frac{1}{16}[(1-\lambda^2)(|v|^2+|v_*|^2)]^2\\
        & \ -\frac{1}{8} (1-\lambda^2)(1+\lambda^2) (|v|^{2}+|v_*|^{2}) |v||v_*|.
    \end{split}    
    \end{equation}
    By using $ \psi(x) = \psi_{1}(x)$ as in \eqref{psionetwo} and noting $ v_{+} = v + v_{*} $ and $ v_{-} = v - v_{*} $, we observe, for some positive constant $C_{\kappa,\lambda}$
    \begin{equation}\label{pp2}
    \begin{split}
        \psi''(2Y) = \left(1 + \frac{\kappa}{2}\right)\frac{\kappa}{2}(2Y)^{\frac{\kappa}{2}-1} &=  \ \left(1 + \frac{\kappa}{2}\right)\frac{\kappa}{2}\left(\frac{|v_{+}|^{2} + \lambda^{2}|v_{-}|^{2}}{2}\right)^{\frac{\kappa}{2}-1} \\[3pt]
       &= C_{\kappa,\lambda}    (|v|^{\kappa -2} + ... + |v_*|^{\kappa -2}),
    \end{split}
    \end{equation} 
    where $\psi(x) = \psi_{2}(x)$ can be substituted for a similar estimate.
    Hence, combining \eqref{pp1} and \eqref{pp2}, we have
    \begin{multline*}
           (Y^2 - Z^2) \psi''(2Y)\ge c_{2,1} (|v|^{\kappa+2}|v_*|^2+|v_*|^{\kappa+2}|v|^2) \\- c_{2,2} (|v|^{\kappa+1}|v_*|+|v_*|^{\kappa+1}|v|),
    \end{multline*}
     where $c_{2,1}$ and $c_{2,2}$ are some positive constants that depends only on $\lambda$ and $\kappa$. Finally, following the similar argument as in \eqref{-Hn1}, we obtain 
      \begin{equation}\label{Hn2}
        H_{n}^{2}
        \geq  C_{2,1}(|v|^{\kappa+2}|v_*|^2+|v_*|^{\kappa+2}|v|^2) - C_{2,2}  (|v|^{\kappa+1}|v_*|+|v_*|^{\kappa+1}|v|),
    \end{equation}
    for some positive constants $C_{2,1}$ and $C_{2,2}$ that  are independent of $n$.

    Together with \eqref{Hn1} and \eqref{Hn2}, we conclude that
\begin{multline}\label{Hn}
        H_{n} \le  C_{1,1} \left(1+\frac{\kappa}{2}\right) \Big[ |v|^2 \left\langle v_{*}\right\rangle^{\kappa} + |v_*|^2 \left\langle v\right\rangle^{\kappa} \Big]\\[3pt] 
        - C_{2,1}(|v|^{\kappa+2}|v_*|^2+|v_*|^{\kappa+2}|v|^2) + C_{2,2}  (|v|^{\kappa+1}|v_*|+|v_*|^{\kappa+1}|v|)\\[3pt]
        \le - C_1(|v|^{\kappa+2}|v_*|^2+|v_*|^{\kappa+2}|v|^2)  +C_2 (\langle v\rangle ^{\kappa+1}|v_*|+\langle v_*\rangle^{\kappa+1}|v|),
    \end{multline}
    for some positive constants $C_1$ and $C_2$ that are independent of $n$.

    Regarding the estimate of $G_n$, for $ \psi(x) = \psi_{1}(x)$ as in \eqref{psionetwo}, by noting that $ Z \leq Y $, we further have 
	\begin{multline*}
				  Z^{2}\int_{0}^{\eta_{0}} \left[ \sqrt{\eta_{0}^{2}-\eta^{2}} -\eta\cos^{-1}\left(\frac{\eta}{\eta_{0}}\right) \right] \\
      \times \left[ \psi''(Y + Z(\cos\beta B+\eta)) + \psi''(Y + Z(\cos\beta B-\eta)) \right] \,\rd \eta\\[2pt]
				 \leq  Z^{2} \eta^{2}_{0} \left( 1+2Y \right)^{\frac{\kappa}{2}-1} \frac{\kappa}{2}\left(1+\frac{\kappa}{2}\right) \int_{\eta=0}^{\eta=\eta_{0}} \rd \left(\frac{\eta}{\eta_{0}}\right)\left[ \sqrt{1 - \left(\frac{\eta}{\eta_{0}}\right)^{2}} -\left(\frac{\eta}{\eta_{0}}\right)\cos^{-1}\left(\frac{\eta}{\eta_{0}}\right) \right] \\
				 \qquad \times \left[ \left(  \frac{Y+Z\cos\beta B}{1+2Y} + \frac{Z}{1+2Y} \eta  \right)^{\frac{\kappa}{2}-1} + \left(  \frac{Y+Z\cos\beta B}{1+2Y} - \frac{Z}{1+2Y} \eta  \right)^{\frac{\kappa}{2}-1} \right] \\[2pt]
				 \leq  Z^{2} \eta^{2}_{0} \left( 1+2Y \right)^{\frac{\kappa}{2}-1} \frac{\kappa}{2}\left(1+\frac{\kappa}{2}\right) 
                \int_{0}^{1} \left[ \sqrt{1-\tilde{\eta}^{2}} -\tilde{\eta}\cos^{-1}\tilde{\eta} \right]\qquad \qquad \qquad \qquad \qquad  \\
                 \qquad \qquad \qquad  \qquad \qquad  \qquad \qquad \qquad \qquad \times \left[(1+\eta_{0}\tilde{\eta})^{\frac{\kappa}{2}-1} + (1-\eta_{0}\tilde{\eta})^{\frac{\kappa}{2}-1} \right] \,\rd \tilde{\eta}\\[2pt]
				\lesssim\  \bigg\{ 
				\begin{aligned}
					 Z^{2} \eta^{2}_{0}, \qquad \qquad \qquad \qquad\text{if}\ \kappa < 2; \\
					 Z^{2} \eta^{2}_{0}\left( 1+2Y \right)^{\frac{\kappa}{2}-1}, \qquad \text{if}\ \kappa \geq 2.
				\end{aligned} \qquad \qquad \qquad \qquad \qquad \qquad \qquad \qquad \qquad 
			\end{multline*}
        where the integrable property is used in the last inequality, and it is noted that a similar derivation also works for $\psi = \psi_2(x)$ as in \eqref{psionetwo} by substituting the specific form of $\psi_2$.
   
		By noticing that $ Z = \left(\lambda|v_{+}||v_{-}|\right)/2 $ and that $ \eta_{0} $ is defined as in \eqref{eta0}, we observe that
		\begin{equation*}
			\begin{split}
				Z^{2} \eta_{0}^{2} &= \left(\frac{|v_{+}||v_{-}|\sin\beta}{2}\right)^{2} \left[ a_{+}^{2} -(\lambda\cos\chi-a_{-})^{2} \right]\\
				&= \left(\frac{|v_{+} \times v_{-}|}{2}\right)^{2} \left[ a_{+}^{2} - a_{+}^{2}\cos^{2}\theta \right]\\
				&= a_{+}^{2} |v_{*} \times v |^{2} \sin^{2}\theta\\
				&\le a_{+}^{2} |v|^{2} |v_{*} |^{2} \sin^{2}\theta,
			\end{split}
		\end{equation*}
		where in the second equality, we notice that $ \beta $ is the angle between $ v_{+} $ and $ v_{-} $ and also we used the relationship \eqref{Lo1}. 
		Consequently, we have the following estimate for $ G_{n}(v,v_{*}) $ that,
		\begin{equation*}
			G_{n}(v,v_{*})\left\{ 
			\begin{aligned}
				&\le \ C_{0,3} a_{+}^{2} |v|^{2} |v_{*}|^{2} \int_{0}^{\frac{\pi}{2}} b_{n}(\cos\theta) \sin^{3}\theta \,\rd\theta \leq C_{3} |v|^{2} |v_{*} |^{2}, \ \text{if}\ \kappa < 2; \\
				&\le \ C_{0,4}a_{+}^{2} |v|^{2} |v_{*}|^{2} \left( 1+|v|^{2}+|v_{*}|^{2} \right)^{\frac{\kappa}{2}-1} \int_{0}^{\frac{\pi}{2}}  b_{n}(\cos\theta) \sin^{3}\theta \,\rd\theta \\
        & \qquad\qquad\qquad\qquad \qquad \leq C_{4} \left( |v_{*}|^{2} \left\langle v\right\rangle^{\kappa} + |v|^{2} \left\langle v_{*}\right\rangle^{\kappa} \right), \ \text{if}\ \kappa \geq 2;
			\end{aligned}
			\right.
		\end{equation*}
		where $C_{0,3},\ C_{3},\ C_{0,4},\ C_{4}$ are constants that are independent of $ n $. This completes the proof of the refined Povzner inequality.
	\end{proof}    
 
        We are now equipped with a refined version of the Povzner-type inequality, which can cover the non-cutoff regime with inelastic hard potential interaction up to $\kappa>0.$ Using this, we will provide the well-posedness theory and the moment-creation property in the next sections.

	\section{Existence of the measure-valued solution}
	\label{sec:existence}
	
	In this section, we establish the existence theory of the measure-valued solution to the Cauchy problem \eqref{IB}-\eqref{F0} using the Fourier transform. The main idea is to first prove the well-posedness for the ``cutoff'' model by a fixed point theorem, by means of which, we can further construct a sequence of the approximated solution to the ``non-cutoff'' equation such that the existence of the ``non-cutoff'' solution can be guaranteed by a compactness argument.  
	
	\subsection{Well-posedness theory for the cutoff model}
	\label{sub:noncutoff}
	
	The cutoff model reads
	\begin{equation}\label{IBEcut}
		\partial_{t} F_t(v) = Q_{e}^{n}(F_t,F_t)(v),
	\end{equation}
	with the initial condition defined as a non-negative probability measure $ F_{0}(v) $, where $ Q_{e}^{n} $ is obtained by replacing the collision kernel $ B(|v-v_{*}|,\sigma) = b(\hat{v}_-\cdot\sigma)\Phi(|v-v_{*}|) $ of \eqref{Bb} by its ``cutoff'' counterpart. More precisely, the angular part $ b(\hat{v}_-\cdot\sigma) $ is being replaced by $ b_n(\hat{v}_-\cdot\sigma) $ as follows:
	\begin{equation}\label{bn}
		b_n(\hat{v}_-\cdot\sigma) \eqdef \min \left\lbrace b\left(\hat{v}_- \cdot \sigma\right), n \right\rbrace \leq b\left(\hat{v}_- \cdot \sigma\right), \quad n\in\mathbb{N},
	\end{equation}
	and the kinetic part $ \Phi(|v-v_{*}|) $ is replaced by $ \Phi_{n}(|v-v_{*}|) $, which is defined as 
	\begin{equation}\label{Pn}
		\Phi_{n}(|v-v_{*}|) \eqdef \Phi(|v-v_{*}|) \phi_{n}(|v-v_{*}|) \eqdef  |v-v_{*}|^{\gamma}\phi_{c}\left(\left|\frac{v-v_{*}}{n}\right|\right),
	\end{equation}
	where $ \phi_{c}\left(|x|\right) \in C_{c}^{\infty}\left(\mathbb{R}^{3}\right) $ satisfying that
	\begin{itemize}
		\item $ \phi_{c} $ is supported on $ \left\lbrace x: |x| \leq 2 \right\rbrace $ with $ |\phi_{c}| \leq 1 $, and
		\item $ \phi_{c} = 1 $ on $ \left\lbrace x: |x| \leq 1 \right\rbrace $,
	\end{itemize}
	such that both $ b_n $ and $ \Phi_{n} $ will approach to the non-cutoff kernel $ b $ and $ \Phi $ as $ n \rightarrow \infty $.
    Throughout the paper, we denote the Fourier(-in-$v$) transform of $\Phi_n$ as $\hat{\Phi}_n$, which is defined as
    $$\hat{\Phi}_n(\zeta)=\int_{\mathbb{R}^3}\Phi_n(|v-v_*|)e^{-iv\cdot \zeta}\,\rd v.$$ 
    Then regarding the Fourier transform, we have the following preliminary lemma:
    
    \begin{lemma}[Lemma 2.5 of \cite{morimoto2016measure}]\label{lemma2.5}
    Let $\Phi_n$ be in \eqref{Pn}. For the hard potential case $\gamma>0,$ we have
    $$\forall k\ge 0,\ k\in\mathbb{N}, \ |\partial_\zeta^k\hat{\Phi}_n(\zeta)|\lesssim_n \frac{1}{\langle\zeta\rangle^{3+\gamma+k}}.$$
    \end{lemma}

  	On the other hand, by applying the Fourier(-in-$v$) transform to \eqref{IBEcut}, the Fourier transform $\varphi=\varphi(t,\xi)\eqdef \mathcal{F}(F_t)(\xi)$ of $F_t(v)$ satisfies the so-called ``Bobylev identity" of inelastic version, see \cite{alexandre2000entropy, Bobylev1975, Bobylev1988, Qi21_soft}, 
	\begin{multline}\label{IBE}
		\partial_{t} \varphi(t, \xi) \\= \int_{\Sd^{2}} b_n\left(\frac{\xi\cdot\sigma}{|\xi|}\right) \int_{\bR^{3}} \hat{\Phi}_{n}(\zeta) \left[ \varphi(t, \xie^{+}-\zeta)\varphi(t, \xie^{-}+\zeta) - \varphi(t, \zeta)\varphi(t, \xi-\zeta) \right] \rd\zeta \rd\sigma,
	\end{multline}
	where
	\begin{equation}\label{xie+}
		\xie^{+} = \left( \frac{1}{2} + \frac{\am}{2} \right)\xi + \frac{\ap}{2}|\xi|\sigma,
	\end{equation}
	\begin{equation}\label{xie-}
		\xie^{-} = \left( \frac{1}{2} - \frac{\am}{2} \right)\xi - \frac{\ap}{2}|\xi|\sigma.
	\end{equation}
	By reformulating \eqref{IBE}, we have
	\begin{equation}\label{phiequ}
		\partial_{t} \varphi(t, \xi) + S\varphi(t, \xi) = \mathcal{G}^{e}_{1} [\varphi] \left(t, \xi\right) + \mathcal{G}^{e}_{2}[\varphi] \left(t, \xi\right),
	\end{equation}
	where $ S = \sup\limits_{q\in\mathbb{R}^{3}}\left|\Phi_{n}\left(\left|q\right| \right) \right| $ and 
	\begin{align}
		\mathcal{G}^{e}_{1} [\varphi] \left(t, \xi\right)  &=  \int_{\Sd^{2}} b_n \left(\frac{\xi\cdot\sigma}{|\xi|}\right) \int_{\bR^{3}} \hat{\Phi}_{n}(\zeta) \left[ \varphi(t, \xie^{+}-\zeta)\varphi(t, \xie^{-}+\zeta)\right]\,\rd\zeta \,\rd\sigma, \label{G1}\\
		\mathcal{G}^{e}_{2} [\varphi] \left(t, \xi\right)  &= S\varphi(t, \xi)-\int_{\Sd^{2}} b_n \left(\frac{\xi\cdot\sigma}{|\xi|}\right) \int_{\bR^{3}} \hat{\Phi}_{n}(\zeta) \left[\varphi(t, \zeta)\varphi(t, \xi-\zeta) \right]\,\rd\zeta \,\rd\sigma. \label{G2}
	\end{align} 
Then	we can further obtain the following integral form of the solution $ \varphi(t, \xi) $,
	\begin{equation}\label{phi_operator}
		\varphi(t, \xi) = \e^{-St}\varphi_{0}(\xi) + \int_{0}^{t} \e^{-S(t-\tau)} \Big( \mathcal{G}^{e}_{1}[\varphi] \left(\tau, \xi\right) + \mathcal{G}^{e}_{2}[\varphi] \left(\tau, \xi\right)  \Big) \,\rd\tau,
	\end{equation}
	where $ \varphi_{0} $ is the Fourier transform of the initial condition $ F_{0} $ in \eqref{F0}.
	
	The proof of the well-posedness of the cutoff equation with the hard potential above can directly follow the counterpart of the soft-potential case \cite[Theorem 3.1]{Qi21_soft}, as long as one can prove the following compactness lemma (Lemma \ref{GG}) of $ \mathcal{G}^{e}_{1} $ and $ \mathcal{G}^{e}_{2} $ in the case of hard potential. Before we introduce the lemma, we define the following space $\mathcal{K}$ and $\mathcal{K}^\alpha$ of characteristic functions:
	$$\mathcal{K}\eqdef \mathcal{F}(P_0(\mathbb{R}^3)) = \left\{\int_{\mathbb{R}^3} \e^{-\im v\cdot \zeta}\,\rd F(v); \ F\in P_0(\mathbb{R}^3) \right\},$$ and
	\begin{equation}\notag
		\K^{\alpha} = \left\lbrace \p\in\K; \ \left\| \p-1 \right\|_{\alpha}= \sup_{0 \neq \xi\in\mathbb{R}^{3}} \frac{\left| \p(\xi)-1 \right|}{|\xi|^{\alpha}} < \infty \right\rbrace.
	\end{equation}
    Note that $\mathcal{K}^\alpha$ is a subspace of the characteristic function space $ \mathcal{K} $ and is a complete metric space endowed with the $C^{0,\alpha}$-distance
	\begin{equation}\notag
		\|\p - \tilde{\p} \|_{\alpha} \eqdef \sup\limits_{0 \neq \xi\in\mathbb{R}^{3}} \frac{|\p(\xi) - \tilde{\p}(\xi)|}{|\xi|^{\alpha}}.
	\end{equation} 
 
    It follows from \cite[Lemma 3.12]{cannone2010infinite} that $\{1\}\subset\mathcal{K}^\alpha\subset \mathcal{K}^0=\mathcal{K}$ for $2\ge \alpha \ge 0,$ and $\mathcal{K}^\alpha=\{1\}$ for all $\alpha>2.$ Note that a characteristic function $\varphi$ satisfies $\varphi(0)=1$ and $|\varphi|\le 1$.
    See \cite{cannone2010infinite,morimoto2016measure,Qi21_soft} for more details about the space of characteristic function.

	\begin{lemma}\label{GG}
		For any restitution coefficient $ e\in (0,1] $ and characteristic function $ \p \in \mathcal{K} $, both of $ \mathcal{G}^{e}_{1}[\varphi] $ and $ \mathcal{G}^{e}_{2}[\varphi] $, defined by \eqref{G1} and \eqref{G2} respectively, are continuous and positive-definite. Furthermore, if $ 0<\gamma\le 2 $ and $ 0< \alpha \leq \min\{\gamma,1\} $, 
		then for any characteristic functions $ \varphi, \tilde{\p} \in \mathcal{K}^{\alpha}$, there exists a constant $ C_{e,n}>0 $ such that,
		\begin{equation}\label{GG12}
			\Big|\mathcal{G}^{e}_{1}[\varphi] + \mathcal{G}^{e}_{2}[\varphi] - \mathcal{G}^{e}_{1}[\tilde{\p}] - \mathcal{G}^{e}_{2}[\tilde{\p}] \Big| \leq \left( S+C_{e,n}\right) \left\| \varphi - \tilde{\p} \right\|_{\alpha} \left| \xi\right|^{\alpha},
		\end{equation}
		for all $ \xi\in\mathbb{R}^{3} $, where we define $C^{0,\alpha}$-distance $\|\cdot\|_\alpha$ as
	\begin{equation}\notag
		\|\p - \tilde{\p} \|_{\alpha} \eqdef \sup\limits_{0 \neq \xi \in\mathbb{R}^{3}} \frac{|\p(\xi) - \tilde{\p}(\xi)|}{|\xi|^{\alpha}}.
	\end{equation}
	\end{lemma}
	\begin{proof}
		
[\textit{Continuity and positive-definiteness}]	For the proof of the continuity and positive-definite (see \cite[Theorem 2.1]{morimoto2016measure} for the definition of positive definiteness) property of $ \mathcal{G}^{e}_{1}[\varphi] $ and $ \mathcal{G}^{e}_{2}[\varphi] $, it suffices to show that $ \mathcal{G}^{e}_{1}[\varphi] $ and $ \mathcal{G}^{e}_{2}[\varphi] $ are characteristic functions by using the Bochner's Theorem \cite[Theorem 3.3]{cannone2010infinite}. For $ \mathcal{G}^{e}_{1}[\varphi] $, we note that $\hat{\Phi}_n$ is integrable with respect to $\zeta$ in \eqref{G1}, and hence the rest of the proof follows from the construction of a mollified characteristic function $\mathcal{G}^{e,m}_{1}[\varphi]$ and the use of the Lebesgue dominated convergence theorem  as in \cite[Lemma 3.4]{Qi21_soft}, \cite[Lemma 2.6]{morimoto2016measure}, and \cite[Lemma 2.1]{PTfourier1996}.  On the other hand, $ \mathcal{G}^{e}_{2}[\varphi] $ in \eqref{G2} is defined the same as its elastic counterpart in \cite[Lemma 2.6]{morimoto2016measure}, and the same proof is applied for $0<\gamma\le 2$ and $0<\alpha\le \min\{\gamma,1\}.$ 
		\newline
		\newline
    \noindent[\textit{Proof of the estimates \eqref{GG12}}] To prove the estimate \eqref{GG12} in hard potential case ($ 0<\gamma\le 2 $), we begin with substituting $ \varphi $ and $ \tilde{\p} $ into \eqref{G1}-\eqref{G2} and taking the subtraction and obtain
		\begin{equation}\notag
			\begin{split}
				&\Big|\mathcal{G}^{e}_{1}[\varphi] + \mathcal{G}^{e}_{2}[\varphi] - \mathcal{G}^{e}_{1}[\tilde{\p}] - \mathcal{G}^{e}_{2}[\tilde{\p}] \Big| \\
				\leq & \int_{\Sd^{2}} b_{n}\left(\frac{\xi\cdot\sigma}{|\xi|}\right) \int_{\bR^{3}} \left| \hat{\Phi}_{n}(\zeta - \xie^{-}) - \hat{\Phi}_{n}(\zeta)\right|  \Big| \varphi(\zeta)\varphi(\xi-\zeta) - \tilde{\p}(\zeta)\tilde{\p}(\xi-\zeta)\Big| \,\rd\zeta \,\rd\sigma\\
				& \qquad+ S \left\| \varphi - \tilde{\p} \right\|_{\alpha} \left| \xi\right|^{\alpha},
			\end{split}
		\end{equation}
		where we utilize the change of variable $ \xie^{-} + \zeta \rightarrow \zeta $ and the fact that $ \xie^{+} + \xie^{-} = \xi $.
        Then, we observe that
		\begin{equation}\notag
			\begin{split}
				&\left| \varphi(\zeta)\varphi(\xi-\zeta) - \tilde{\p}(\zeta)\tilde{\p}(\xi-\zeta)\right|\\
				\leq & \left| \varphi(\zeta)\varphi(\xi-\zeta) - \varphi(\zeta)\tilde{\p}(\xi-\zeta) + \varphi(\zeta)\tilde{\p}(\xi-\zeta) - \tilde{\p}(\zeta)\tilde{\p}(\xi-\zeta) \right|\\
				\leq &\left| \varphi(\zeta) \right| \left\|\varphi -\tilde{\p} \right\|_{\alpha} \left|\xi-\zeta \right|^{\alpha} + \left\|\varphi -\tilde{\p} \right\|_{\alpha} \left| \zeta \right|^{\alpha} \left|\tilde{\p}(\xi-\zeta)\right|\\[2pt]
				\leq & \frac{\left|\xi-\zeta \right|^{\alpha} + \left| \zeta \right|^{\alpha}}{\left| \xi \right|^{\alpha}} \left\|\varphi -\tilde{\p} \right\|_{\alpha} \left| \xi \right|^{\alpha},
			\end{split}
		\end{equation}
		where we used the property of the characteristic functions that $ \left| \varphi(\zeta) \right| < 1 $ and $ \left|\tilde{\p}(\xi-\zeta)\right| < 1 $ for the last inequality.
		Hence, combining the analysis above, we realize that it suffices to prove the following estimate:
		\begin{equation}\label{desired est}
			\int_{\Sd^{2}} b_{n}\left(\frac{\xi\cdot\sigma}{|\xi|}\right) \int_{\bR^{3}} \left| \hat{\Phi}_{n}(\zeta - \xie^{-}) - \hat{\Phi}_{n}(\zeta)\right|  \frac{\left|\xi-\zeta \right|^{\alpha} + \left| \zeta \right|^{\alpha}}{\left| \xi \right|^{\alpha}}  \,\rd\zeta \,\rd\sigma \leq C_{e,n}.
		\end{equation}
		In fact, for $ \left| \xi \right| \leq 1 $, by considering Lemma \ref{lemma2.5} and the following estimates
		\begin{equation}\label{xie-xi}
			\frac{|\xie^{-}|}{|\xi|^{\alpha}} = \frac{a_+ \sin\frac{\bar{\theta}}{2}|\xi|}{|\xi|^{\alpha}} \leq a_+ \sin\frac{\bar{\theta}}{2},
		\end{equation}
		with $\bar{\theta}$ the angle between $\frac{\xi}{|\xi|}$ and $\sigma$, and
		\begin{equation}\label{zeta+xi}
			\begin{split}
				|\xi - \zeta|^{\alpha} + |\zeta|^{\alpha} &\lesssim \left|\zeta - \tau\xie^{-} \right|^{\alpha} + 1,\ \text{ for any } \tau\in[0,1],
			\end{split}
		\end{equation}
		we obtain that
		\begin{equation}\notag
			\begin{split}
				&\int_{\bR^{3}} \left| \hat{\Phi}_{n}(\zeta - \xie^{-}) - \hat{\Phi}_{n}(\zeta)\right|  \frac{\left|\xi-\zeta \right|^{\alpha} + \left| \zeta \right|^{\alpha}}{\left| \xi \right|^{\alpha}} \,\rd\zeta \\
				\le \ & \int_{\bR^{3}} \int_{0}^{1} \left| \frac{\partial \hat{\Phi}_{n}}{\partial\zeta}(\zeta -\tau \xie^{-}) \right| \,\rd\tau \left| \xie^{-} \right| \frac{\left|\xi-\zeta \right|^{\alpha} + \left| \zeta \right|^{\alpha}}{\left| \xi \right|^{\alpha}}  \,\rd\zeta\\
				\lesssim_n & a_+ \sin\frac{\bar{\theta}}{2} \int_{0}^{1} \int_{\bR^{3}} \frac{|\xi - \zeta|^{\alpha} + |\zeta|^{\alpha} }{\left\langle \zeta - \tau \xie^{-} \right\rangle^{3+\gamma+1} } \,\rd\zeta \,\rd\tau\\
				\lesssim_n & a_+ \sin\frac{\bar{\theta}}{2} \int_{0}^{1} \int_{\bR^{3}} \frac{\left|\zeta - \tau\xie^{-} \right|^{\alpha} + 1 }{\left\langle \zeta - \tau \xie^{-} \right\rangle^{3+\gamma+1} } \,\rd\zeta \,\rd\tau \lesssim_n a_+ \sin\frac{\bar{\theta}}{2},
			\end{split}
	\end{equation}
	where we utilize Lemma \ref{lemma2.5} and \eqref{xie-xi} in the second inequality as well as \eqref{zeta+xi} in the third inequality.
		On the other hand, for $ |\xi| > 1 $, we use Lemma \ref{lemma2.5} and obtain
		\begin{equation}\notag
			\begin{split}
				&\int_{\bR^{3}} \left| \hat{\Phi}_{n}(\zeta - \xie^{-}) - \hat{\Phi}_{n}(\zeta)\right|  \frac{\left|\xi-\zeta \right|^{\alpha} + \left| \zeta \right|^{\alpha}}{\left| \xi \right|^{\alpha}} \,\rd\zeta \\
				\leq & \int_{\bR^{3}} \left| \hat{\Phi}_{n}(\zeta - \xie^{-}) \right|  \frac{\left|\xi-\zeta \right|^{\alpha} + \left| \zeta \right|^{\alpha}}{\left| \xi \right|^{\alpha}} \,\rd\zeta + \int_{\bR^{3}} \left| \hat{\Phi}_{n}(\zeta)\right|  \frac{\left|\xi-\zeta \right|^{\alpha} + \left| \zeta \right|^{\alpha}}{\left| \xi \right|^{\alpha}} \,\rd\zeta\\
				\lesssim_n & \int_{\bR^{3}}  \frac{2\left|\zeta - \xie^{-}\right|^{\alpha} + 2}{\left\langle \zeta - \xie^{-} \right\rangle^{3+\gamma} } \,\rd\zeta + \int_{\bR^{3}}  \frac{2\left|\zeta\right|^{\alpha} + 1}{\left\langle \zeta \right\rangle^{3+\gamma} } \,\rd\zeta \lesssim_n 1.
			\end{split} 
		\end{equation}
		This completes the proof of the desired estimate \eqref{desired est} and the proof of the lemma.	
    \end{proof}
	
    Consequently, we use Lemma \ref{GG} and apply the standard Banach fixed-point theorem and continuation argument towards \eqref{phi_operator} in the space $ \mathcal{K}^{\alpha} $. After applying the inverse Fourier transform $F^n_t=\mathcal{F}^{-1}(\varphi^n)$, we obtain the following well-posedness result for the cutoff model:
    
    \begin{proposition}\label{cutoff_existence}
    Let $ e\in(0,1] $ and $ \alpha_{0}\in \left(0,2\right] $ and let the collision kernel $B$ be a cutoff kernel $ B = b_{n} \Phi_{n} $. Then, for any initial datum $ F_{0}(v) \in P_{\alpha_{0}}(\mathbb{R}^{3}) $, there exists a unique solution $ F^n_{t}(v) \in C\left( \left[0,\infty \right), P_{\alpha}(\mathbb{R}^{3}) \right) $ with $ \alpha \in (0,\alpha_{0}) $ to the cutoff model \eqref{IBEcut} in the measure-valued sense. Moreover, if $F_{0}(v) \in P_{2}(\mathbb{R}^{3})$ (resp., if $F_{0}(v) \in P_{\beta}(\mathbb{R}^{3})$ for $\beta \geq 1$), then the solution $F^n_{t}(v)$ satisfies
    \begin{equation}\label{dissFt.cut}
     \int_{\bR^{3}} |v|^{2} \,\rd F^n_{t}(v) \leq \int_{\bR^{3}} |v|^{2} \,\rd F_{0}(v)\ (resp., \int_{\bR^{3}} v_j \,\rd F^n_{t}(v) = \int_{\bR^{3}} v_j \,\rd F_{0}(v), \ j =1,2,3),
    \end{equation}
     for any $t>0$.
    \end{proposition}
    
    \begin{remark}
    Note that, mass conservation of the solution is naturally inherited from the definition of the measure-valued solution, since $1 = \psi \in C_b^2(\mathbb{R}^3)$; however, the momentum conservation and energy dissipation of the obtained solution, i.e., \eqref{dissFt.cut}, have to be further checked, which can be done by using the same technique as in the proof of {\rm \cite[Theorem 2.8]{morimoto2016measure}}.
    \end{remark}

        In addition, we can also obtain the following moments-estimate of the sequence of the measure-valued solutions $ F^{n}_{t}(v) $ as a direct application of the Povzner estimate that we proved in Proposition~\ref{Povzner}, if the initial datum $ F_{0}(v) $ has finite $ (2+\kappa) $-moments for any given $\kappa>0$.
        
    \begin{corollary}\label{Proposition2+kappa}
		Assume that $ F_{0} $ satisfies, for a given $ \kappa > 0 $,
		\begin{equation}\notag
			\int_{\bR^{3}} \left\langle v \right\rangle^{2+\kappa} \,\rd F_{0}(v) < \infty.
		\end{equation}
        Then, for any fixed $ T > 0 $, there exists constants $C_1$ and $ C_{\kappa,e,T} > 0 $ independent of $ n $ such that,
		\begin{multline}\label{MomentProduction_n1}
			\sup_{t \in [0,T]} \int_{\bR^{3}} \left\langle v \right\rangle^{2+\kappa} \,\rd F_{t}^{n}(v) + C_1\int_{0}^{t} \left( \int_{\bR^{3}} \left\langle v \right\rangle^{2+\kappa} \min\{ \left\langle v \right\rangle^{\gamma}, n^{\gamma} \} \,\rd F_{\tau}^{n}(v) \right) \\
     \times \left( \int_{\bR^{3}} \left\langle v_* \right\rangle^{2} \min\{ \left\langle v_* \right\rangle^{\gamma}, n^{\gamma} \} \,\rd F_{\tau}^{n}(v_*) \right) \,\rd \tau
        \leq C_{\kappa,e,T}.
		\end{multline}
	\end{corollary}
	
	\begin{proof}
	By considering the cutoff equation \eqref{IBEcut} and integrating over $ v \in \bR^{3} $ with the test function $ \psi(x) = \psi_{2}(x) = (1+x)^{1+\frac{\kappa}{2}} -1 $ as well as the conservation of mass, we observe that
		\begin{multline*}
			 \int_{\bR^{3}} \left\langle v \right\rangle^{2+\kappa} \,\rd F_{t}^{n}(v)- \int_{\bR^{3}} \left\langle v \right\rangle^{2+\kappa} \,\rd F_{0}^{n}(v)\\ 
        = \frac{1}{2} \int_0^t\int_{\bR^{3}} \int_{\bR^{3}} \Phi_{n}(|v-v_{*}|) K_{n}^{e}(v,v_{*}) \,\rd F_{\tau}^{n}(v) \,\rd F_{\tau}^{n}(\vs)\rd \tau,
		\end{multline*}
		where $ K_{n}^{e}(v,v_{*}) $ is defined as in \eqref{Ke}. Now, we use the Povzner inequality (Proposition \ref{Povzner} with \eqref{Hboundbetter}) and obtain
		\begin{equation}\label{middlestep}
			\begin{split}
				&\int_{\bR^{3}} \left\langle v \right\rangle^{2+\kappa} \,\rd F_{t}^{n}(v) + C_1 \int_{0}^{t} \left( \int_{\bR^{3}} \left\langle v \right\rangle^{2+\kappa} \min\{ \left\langle v \right\rangle^{\gamma}, n^{\gamma} \} \,\rd F_{\tau}^{n}(v) \right) \\
    &\qquad \qquad \qquad \qquad \qquad\qquad \qquad \times \left( \int_{\bR^{3}} \left\langle v_* \right\rangle^{2} \min\{ \left\langle v_* \right\rangle^{\gamma}, n^{\gamma} \} \,\rd F_{\tau}^{n}(v_*) \right) \,\rd \tau \\
			&\leq \int_{\bR^{3}} \left\langle v \right\rangle^{2+\kappa} \,\rd F_{0}(v) \\
    &\qquad+ C_{3,4} \left( \int_{\bR^{3}} \left\langle v \right\rangle^{2} \,\rd F_{0}(v) \right) \int_{0}^{t}\int_{\bR^{3}} \left\langle v \right\rangle^{\max\{2,\kappa\}} \min\{ \left\langle v \right\rangle^{\gamma}, n^{\gamma} \} \,\rd F_{\tau}^{n}(v) \,\rd \tau\\
    & + C_2 \int_{0}^{t} \left( \int_{\bR^{3}} \left\langle v \right\rangle^{1+\kappa} \min\{ \left\langle v \right\rangle^{\gamma}, n^{\gamma} \} \,\rd F_{\tau}^{n}(v) \right)  \left( \int_{\bR^{3}} |v_*| \min\{ \left\langle v_* \right\rangle^{\gamma}, n^{\gamma} \} \,\rd F_{\tau}^{n}(v_*) \right) \,\rd \tau \\
    \leq & \int_{\bR^{3}} \left\langle v \right\rangle^{2+\kappa} \,\rd F_{0}(v) \\
    &\qquad+ C_{5} \left( \int_{\bR^{3}} \left\langle v \right\rangle^{2+\kappa} \,\rd F_{0}(v) \right) \int_{0}^{t}\int_{\bR^{3}} \left\langle v \right\rangle^{\max\{2,\kappa\}} \min\{ \left\langle v \right\rangle^{\gamma}, n^{\gamma} \} \,\rd F_{\tau}^{n}(v) \,\rd \tau.
			\end{split}
		\end{equation}
  where the constants $C_1$,$C_2$ are the constants of Proposition \ref{Povzner} and $C_{3,4}$ depends only on the constants $C_3$ and $C_4$ of Proposition \ref{Povzner}. 
	Now, we split the domain $v\in \mathbb{R}^3$ into $|v|\le R_0$ and $|v|\ge R_0$ for a sufficiently large $R_0>0.$ 	By further selecting $ R_{0} > 0 $ large such that 
 $$ C_1\left(1+R_0^2\right)^{\frac{1}{2}\min\{2,\kappa\}} \geq 2C_{5} \int_{\bR^{3}} \left\langle v \right\rangle^{2+\kappa} \,\rd F_{0}(v) ,$$ 
		we obtain
	\begin{equation}\notag
			\begin{split}
				&C_1 \int_{|v| \geq R_{0}} \left\langle v \right\rangle^{2+\kappa} \min\{ \left\langle v \right\rangle^{\gamma}, n^{\gamma} \} \,\rd F_{\tau}^{n}(v) \\
				&\ge 2 C_{5} \left( \int_{\bR^{3}} \left\langle v \right\rangle^{2+\kappa} \,\rd F_{0}(v) \right) \int_{|v|\ge R_0} \left\langle v \right\rangle^{\max\{2,\kappa\}} \min\{ \left\langle v \right\rangle^{\gamma}, n^{\gamma} \} \,\rd F_{\tau}^{n}(v),
			\end{split}
		\end{equation}
		and
		\begin{equation}\notag
			\begin{split}
				& C_{5} \left( \int_{\bR^{3}} \left\langle v \right\rangle^{2+\kappa} \,\rd F_{0}(v) \right) \int_{|v|\le R_0} \left\langle v \right\rangle^{\max\{2,\kappa\}} \min\{ \left\langle v \right\rangle^{\gamma}, n^{\gamma} \} \,\rd F_{\tau}^{n}(v)\lesssim R_0^{2+\kappa+\gamma}\lesssim 1.
			\end{split}
		\end{equation}
		Thus, they yield \eqref{MomentProduction_n1} together with \eqref{middlestep}, and this completes the proof.
	\end{proof}
	This closes the discussions on the existence and the moments-estimates under an angular cutoff. In the next subsection, we introduce the rather standard approximation of a non-cutoff model by cutoff models.
	
    \subsection{Non-cutoff model as an approximation of cutoff model}
    \label{sub:cutoff}
	
    To obtain the existence of the solution to the non-cutoff equation, in this section, we will regard the non-cutoff model as an approximation of cutoff models, of which the well-posedness has already been studied in the last subsection.
	In fact, by directly applying Proposition \ref{cutoff_existence}, we can construct a sequence of approximated solutions  $ \{F^{n}_{t}(v)\} = \{\mathcal{F}^{-1}[\p^{n}(t,\xi)]\} $ where $ \{\p^{n}(t,\xi)\} $ is the sequence of solutions to the cutoff equation \eqref{IBEcut} as before.
	
	In order to pass to the limit for approximated solutions by the compactness argument, we need to have the uniform boundedness and equicontinuity  of the sequence of solutions $\{\varphi^{n}\}$.
	In fact, we note that the uniform-boundedness property $|\varphi^n|\le 1$ and the equicontinuity in the Fourier variable $ \xi $ of the sequence of the solution $ \{\varphi^{n}\} $ are directly inherited from the fact that all of them are found in the space of characteristic function. 	Hence, we only need to prove the equicontinuity in the temporal variable $ t $. 
	
	To this end, we observe that for any $ 0 \leq s < t \leq T $, there exist constants $ C_{e,\psi}, C'_{e,\psi} > 0 $ (independent of $ n $) such that, for any $ \psi \in C_{b}^{2}(\mathbb{R}^{3}) $,
	\begin{equation}
		\begin{split}
			& \left| \int \psi \,\rd F_{t}^{n}(v)-\int \psi \,\rd F_{s}^{n}(v) \right|\\
			&\le \frac{1}{2} \int_{s}^{t}  \int_{\bR^{3}} \int_{\bR^{3}} \bigg| \int_{\Sd^{2}} b_{n}(\hat{v}_- \cdot \sigma)\left(\psi'_{*} + \psi' - \psi_{*} - \psi \right) \,\rd \sigma \bigg| \\
            &\qquad\qquad\qquad\qquad\qquad \times  \left|v-v_{*}\right|^{\gamma} \phi_{n}(|v-v_{*}|) \,\rd F_{\tau}^{n}(v) \,\rd F_{\tau}^{n}(v_{*})  \,\rd \tau\\
			&\le \frac{C_{e, \psi}}{2} \int_{s}^{t}  \int_{\bR^{3}} \int_{\bR^{3}}  |v-v_{*}|^{\gamma+2}\phi_{n}(|v-v_{*}|) \,\rd F_{\tau}^{n}(v) \,\rd F_{\tau}^{n}(v_{*})  \,\rd \tau\\
			&\le C'_{e, \psi} \bigg[ \left|t-s\right| \sup_{s\leq\tau\leq t} \left(\int_{\bR^{3}}  \left\langle v \right\rangle^{2}  \,\rd F_{\tau}^{n}(v) \right) \left( \int_{\bR^{3}} \left\langle v_{*} \right\rangle^{2} \,\rd F_{\tau}^{n}(v_{*}) \right)\\
			& + \int_{s}^{t} \int_{\bR^{3}} \int_{\bR^{3}} \chi_{\{|v|\geq 2|v_{*}|\} \cup \{|v_{*}|\geq 2|v|\} } |v-v_{*}|^{2}\\&\qquad\qquad\qquad\qquad\qquad\qquad\times \min\{|v-v_{*}|^{\gamma}, n^{\gamma}\} \,\rd F_{\tau}^{n}(v) \,\rd F_{\tau}^{n}(v_{*}) \,\rd \tau \bigg] \\
			&\le C'_{e, \psi} \Big[ \left|t-s\right| \left( \int_{\bR^{3}} \left\langle v \right\rangle^{2}  \,\rd F_{0} \right)^{2} + \int_{s}^{t} \left( \int_{\bR^{3}} \left\langle v \right\rangle^{2} \min\{ \left\langle v \right\rangle^{\gamma}, n^{\gamma} \} \,\rd F_{\tau}^{n}(v) \right)^2 \,\rd \tau \Big] \label{equihard},
		\end{split}
	\end{equation}
	where $ \chi_{U} (v) \in C_{c}^{\infty}(\mathbb{R}^{3}) $ is the indicator function 	\begin{equation}\label{chi1}
		\chi_{U} (v) \eqdef \begin{cases}
			1,& \text{if}\ v \in U, \\
			0,& \text{if}\ v \notin U,
		\end{cases}
	\end{equation} $\phi_n$ is defined as in \eqref{Pn}, and we utilized the estimate \cite[Eq. (4.15)]{Qi21_soft},
    $$ \bigg| \int_{\Sd^{2}} b_{n}(\hat{v}_- \cdot \sigma)\left(\psi'_{*} + \psi' - \psi_{*} - \psi \right)\,\rd \sigma \bigg| \leq C_{e,\psi}|v-v_{*}|^{2},$$ 
    in the second inequality, the condition $0<\gamma\le 2$ in the first inequality, and the estimate \cite[pp. 880]{morimoto2016measure} in the third inequality.\\
	For the second term in \eqref{equihard}, we take advantage of \eqref{MomentProduction_n1} in Corollary~\ref{Proposition2+kappa} and obtain that for any $ T > 0 $, there exists a $ C_{\kappa,e,T} > 0$ independent of $ n $ for any given $\kappa>0$ such that
	\begin{equation}\notag
		\int_{s}^{t} \left( \int_{\mathbb{R}^{3}} \chi_{\{ |v| \geq R \}} \left\langle v \right\rangle^{2} \min\{ \left\langle v \right\rangle^{\gamma}, n^{\gamma} \} \,\rd F_{\tau}^{n}(v) \right)^2 \,\rd \tau \lesssim \frac{C_{\kappa,e,T}}{R^{\kappa}},
	\end{equation} 
        for a sufficiently large $R\ge 1$, since we observe that
	\begin{multline*}
	    R^{\kappa}	\int_{s}^{t} \left( \int_{\mathbb{R}^{3}} \chi_{\{ |v| \geq R \}} \left\langle v \right\rangle^{2} \min\{ \left\langle v \right\rangle^{\gamma}, n^{\gamma} \} \,\rd F_{\tau}^{n}(v) \right)^2 \,\rd \tau \\ 
	    \le 	\int_{s}^{t} \left( \int_{\mathbb{R}^{3}} \chi_{\{ |v| \geq R \}} \left\langle v \right\rangle^{2+\kappa } \min\{ \left\langle v \right\rangle^{\gamma}, n^{\gamma} \} \,\rd F_{\tau}^{n}(v) \right)^2 \,\rd \tau 
	    \lesssim C_{\kappa,e,T},	\end{multline*} for $|v|\ge R.$
	On the other hand, since $0<\gamma\le 2,$ we also have
	\begin{multline*}
		\int_{s}^{t} \left( \int_{\mathbb{R}^{3}} \chi_{\{ |v| < R \}} \left\langle v \right\rangle^{2} \min\{ \left\langle v \right\rangle^{\gamma}, n^{\gamma} \} \,\rd F_{\tau}^{n}(v) \right)^2 \,\rd \tau\\
		\lesssim	\int_{s}^{t} \left( \int_{\mathbb{R}^{3}} \chi_{\{ |v| < R \}} R^4\,\rd F_{\tau}^{n}(v) \right)^2 \,\rd \tau 
			\lesssim C_{\kappa,e,T}R^{8} |t-s|.
	\end{multline*} 
        Then, by choosing $\psi=\e^{-i\xi\cdot v}$ we obtain the equicontinuity in the time variable $ t $.

	Thus, thanks to the Arzel\`a-Ascoli Theorem, we are able to take the limit of $ \{\varphi^{n}\} $ (up to a subsequence) in the sense that, on every compact subset of $ [0,\infty) \times \mathbb{R}^{3} $,
	\begin{equation}\label{limp}
		\p(t,\xi) = \lim\limits_{n \rightarrow \infty} \p^{n}(t,\xi),
	\end{equation}
	such that the limit function $ \p(t,\xi) $ is a characteristic function as well.

	To obtain the moments-estimate of the measure-valued solution  $$ F_{t}(v) = \mathcal{F}^{-1}[\p(t,\xi)] $$ in the limit $n\to \infty$, we also introduce the following Lemma~\ref{DissiLemma} to illustrate that some moments properties will be maintained in the limiting process. The proof of the following lemma below in the inelastic case is the same as the counterpart \cite[Lemma 3.2]{morimoto2016measure} of the elastic case, since the estimates are only associated with the variables $v$ and $v_*$ in the limiting process, which has no distinction between the elastic and inelastic cases. We omit the proof.
	\begin{lemma}{{\rm\cite[Lemma 3.2]{morimoto2016measure}}}\label{DissiLemma}
		Assume that $ F^{n}_{t}(v) = \mathcal{F}^{-1}\left[\p^{n}(t,\xi)\right] $ is a sequence of measure-valued solution to cutoff equation and $ F_{t}(v) = \mathcal{F}^{-1}[\p(t,\xi)] $ is the limit of sequence with $ \p(t,\xi) $ obtained in \eqref{limp}. Then,\\
		(i) For every $ t > 0 $, we have,
		\begin{equation}\label{dissFt}
	\int_{\bR^{3}} |v|^{2} \,\rd F_{t}(v) \leq \int_{\bR^{3}} |v|^{2} \,\rd F_{0}(v).
		\end{equation}
		(ii) For any $  T >0 $ and $ \psi (v) \in C(\mathbb{R}^{3}) $ with $ |\psi(v)| \lesssim \left\langle v  \right\rangle^{l}$  for some $ 0 < l < 2 $, we have,
		\begin{equation}\label{limFt1}
			\lim\limits_{n\rightarrow \infty} \int_{\bR^{3}} \psi(v) \,\rd F_{t}^{n}(v) = \int_{\bR^{3}} \psi(v) \,\rd F_{t}(v),
		\end{equation} 
		uniformly for $ t\in [0,T] $.\\
		(iii) For any $  T >0 $ and $ \psi (v,\vs) \in C(\mathbb{R}^{3}\times \mathbb{R}^{3}) $ with $ |\psi(v,\vs) | \lesssim \left\langle v \right\rangle^{l} \left\langle \vs \right\rangle^{l}$ for some $ 0 < l < 2 $, we have,
		\begin{equation}\label{limFt2}
			\lim\limits_{n\rightarrow \infty} \int_{\bR^{3}} \int_{\bR^{3}} \psi (v,\vs) \,\rd F_{t}^{n}(v)\,\rd F_{t}^{n}(\vs) = \int_{\bR^{3}} \int_{\bR^{3}} \psi (v,\vs) \,\rd F_{t}(v) \,\rd F_{t}(\vs),
		\end{equation} 
		uniformly for $ t\in [0,T] $.
	\end{lemma}

        \begin{remark}
            Note that the energy boundedness $\int_{\mathbb{R}^3}|v|^2 \,\rd F(v)<\infty$ and the energy dissipation property \eqref{dissFt}  exclude the possible non-uniqueness scenario that is caused by the instantaneous energy jump {\rm\cite{Wennberg_nonuniqueness}} in the elastic hard-sphere model with an angular cutoff.
        \end{remark}

	Based on the Lemma~\ref{DissiLemma}, we finally obtain the moments estimate for the limit $ F_{t}(v) $ if initial datum $ F_{0}(v) $ has finite $ (2+\kappa) $-moments.

	\begin{corollary}\label{Proposition2+kappa.noncut}
		Assume that $ F_{0} $ satisfies, for a given $ \kappa > 0 $,
		\begin{equation}\notag
			\int_{\bR^{3}} \left\langle v \right\rangle^{2+\kappa} \,\rd F_{0}(v) < \infty.
		\end{equation}
		Let $ F_{t}(v) $ satisfy Lemma~\ref{DissiLemma}. Then, for any fixed $ T > 0 $, there exists constants $C_1$ and $ C_{\kappa,e,T} > 0 $ independent of $ n $ such that
		\begin{multline}\label{MomentProduction_1.noncut}
		    \sup_{t \in [0,T]} \int_{\bR^{3}} \left\langle v \right\rangle^{2+\kappa} \,\rd F_{t}(v)\\ + C_1\int_{0}^{T} \left( \int_{\bR^{3}} \left\langle v \right\rangle^{2+\kappa+\gamma}  \,\rd F_{\tau}^{n}(v) \right) \left( \int_{\bR^{3}} \left\langle v_* \right\rangle^{2+\gamma}  \,\rd F_{\tau}^{n}(v_*) \right) \,\rd \tau  \leq C_{\kappa,e,T}.
		\end{multline}
	\end{corollary}
	
	\begin{proof}We obtain \eqref{MomentProduction_1.noncut} by taking the limit of \eqref{MomentProduction_n1} as $n\to \infty$ by means of Lemma~\ref{DissiLemma}. This completes the proof.
	\end{proof}

	Now we are ready to pass to the limit $n \to \infty$ and discuss the non-cutoff case. In the next subsection, prove the main theorem (Theorem \ref{main1}) in the non-cutoff situation.

	\subsection{Existence theory for the non-cutoff model }
	\label{sub:pro_main1}
	
	In this subsection, we will complete the main existence theorem by showing that the limit function $ F_{t}(v) $ is indeed the measure-valued solution to the original non-cutoff equation \eqref{IBE}.  
	
	\begin{proof}[Proof of Theorem \ref{main1}]
		Let $ \Psi^{e}_{n}(v,v_{*})$ and $ \Psi^{e}(v,v_{*}) $ denote 
		\begin{equation*}
			\Psi^{e}_{n}(v,v_{*}) = |v-v_{*}|^{\gamma} \phi_{n}(|v-v_{*}|) \int_{\Sd^{2}} b_{n}(\hat{v}_- \cdot \sigma) \left(\psi'_{*} + \psi' - \psi_{*} - \psi \right) \,\rd \sigma,
		\end{equation*}
		and 
		\begin{equation*}
			\Psi^{e}(v,v_{*}) = |v-v_{*}|^{\gamma} \int_{\Sd^{2}} b(\hat{v}_- \cdot \sigma) \left(\psi'_{*} + \psi' - \psi_{*} - \psi \right) \,\rd \sigma.
		\end{equation*}
  where $\psi \in C_b^2(\mathbb{R}^3) $ is the test function in the definition of the measure-valued solution.\\
		Then, to establish the limit $n\to \infty$ in the measure-valued sense, it suffices to calculate the following difference:
		\begin{multline}\label{symmweak}
				 \int_{0}^{t}\int_{\bR^{3}} \int_{\bR^{3}} \int_{\Sd^{2}}  b_{n}(\hat{v}_- \cdot \sigma) \left| v-v_{*}\right|^{\gamma} \phi_{n}(|v-v_{*}|)\\
            \times  \left(\psi'_{*} + \psi' - \psi_{*} - \psi \right) \,\rd \sigma \,\rd F^{n}_{\tau} (v) \,\rd F^{n}_{\tau} (v_{*}) \,\rd \tau \\
				  \qquad\qquad - \int_{0}^{t}\int_{\bR^{3}} \int_{\bR^{3}} \int_{\Sd^{2}} b(\hat{v}_- \cdot \sigma) \left| v-v_{*}\right|^{\gamma}  \left(\psi'_{*} + \psi' - \psi_{*} - \psi \right) \,\rd \sigma \,\rd F_{\tau} (v) \,\rd F_{\tau} (v_{*}) \,\rd \tau\\[3pt]
				= \int_{0}^{t}\int_{\bR^{3}} \int_{\bR^{3}} \Psi^{e}_{n}(v,v_{*})  \,\rd F^{n}_{\tau}(v) \,\rd F^{n}_{\tau}(v_{*}) \,\rd \tau - \int_{0}^{t}\int_{\bR^{3}} \int_{\bR^{3}} \Psi^{e}(v,v_{*})  \,\rd F_{\tau}(v) \,\rd F_{\tau}(v_{*}) \,\rd \tau\\[3pt]
				\eqdef   I_{1} - I_{2} + I_{3} + I_{4}, \qquad \qquad \qquad \qquad \qquad \qquad \qquad \qquad \qquad \qquad \qquad \qquad \qquad \qquad 
			\end{multline}
		where the integrals $ I_{j}$ with $j =1,2,3,4 $ are defined as follows: for some $R>0$, we define
		\begin{equation*}
			\begin{split}
				I_{1}&\eqdef \int_{0}^{t}\int_{\bR^{3}}\int_{\bR^{3}} \chi_{\{|v|> R\} \cup \{|v_{*}|> R\}} \Psi^{e}_{n}(v,v_{*})  \,\rd F^{n}_{\tau}(v) \,\rd F^{n}_{\tau}(v_{*}) \,\rd \tau,\\
				I_{2}&\eqdef \int_{0}^{t}\int_{\bR^{3}}\int_{\bR^{3}} \chi_{\{|v|> R\} \cup \{|v_{*}|> R\}} \Psi^{e}(v,v_{*})  \,\rd F_{\tau}(v) \,\rd F_{\tau}(v_{*}) \,\rd \tau,\\
				I_{3}&\eqdef \int_{0}^{t} \int_{\bR^{3}}\int_{\bR^{3}} \chi_{\{|v| \leq R\} \cap \{|v_{*}|\leq R\}} \left[ \Psi^{e}_{n}(v,v_{*}) - \Psi^{e}(v,v_{*}) \right]  \,\rd F^{n}_{\tau}(v) \,\rd F^{n}_{\tau}(v_{*}) \,\rd \tau,\\
				I_{4}&\eqdef \int_{0}^{t} \Big[ \int_{\bR^{3}} \chi_{\{|v| \leq R\} \cap \{|v_{*}|\leq R\}} \Psi^{e}(v,v_{*}) \,\rd F^{n}_{\tau}(v) \,\rd F^{n}_{\tau}(v_{*})\\
				& \qquad \qquad \qquad  \qquad - \int_{\bR^{3}} \chi_{\{|v| \leq R\} \cap \{|v_{*}|\leq R\}} \Psi^{e}(v,v_{*}) \,\rd F_{\tau}(v) \,\rd F_{\tau}(v_{*}) \Big] \,\rd \tau.
			\end{split}	
		\end{equation*} 
		Note that
		\begin{multline*}
				|\Psi^{e}_{n}(v,v_{*})| \lesssim  \left\langle v \right\rangle^{2}  \left\langle v_{*} \right\rangle^{2} \chi_{\{ |v|<2|v_{*}|<4|v| \}} \\
				+ \left\langle v \right\rangle^{2} \min\{ \left\langle v \right\rangle^{\gamma}, n^{\gamma}\} \chi_{\{ |v|\geq 2|v_{*}| \}}
				+ \left\langle v_{*} \right\rangle^{2} \min\{ \left\langle v_{*} \right\rangle^{\gamma}, n^{\gamma}\} \chi_{\{ |v_{*}|\geq 2|v| \}},
		\end{multline*}
		and
		\begin{multline*}
				|\Psi^{e}(v,v_{*})| \lesssim |v-v_{*}|^{\gamma+2} \lesssim \left\langle v \right\rangle^{2}  \left\langle v_{*} \right\rangle^{2} \chi_{\{ |v|<2|v_{*}|<4|v| \}} \\
				+ \left\langle v \right\rangle^{\gamma + 2} \chi_{\{ |v|\geq 2|v_{*}| \}} + \left\langle v_{*} \right\rangle^{\gamma + 2} \chi_{\{ |v_{*}|\geq 2|v| \}}.
		\end{multline*}
		By applying the moments-estimates \eqref{MomentProduction_n1} and \eqref{MomentProduction_1.noncut}, we first obtain that
		\begin{equation}\notag
			\begin{split}
				I_{1} &\le \left|  \int_{0}^{t}\int_{\bR^{3}}\int_{\bR^{3}} \chi_{\{|v|> R\} \cup \{|v_{*}|> R\}} \Psi^{e}_{n}(v,v_{*})  \,\rd F^{n}_{\tau}(v) \,\rd F^{n}_{\tau}(v_{*}) \,\rd \tau \right|\\
				\lesssim & T R^{-2\kappa} \left( \int_{\bR^{3}} \left\langle v \right\rangle^{2+\kappa}  \,\rd F_{0} \right)^{2} + R^{-\kappa}\int_{0}^{T} \left( \int_{\bR^{3}} \left\langle v \right\rangle^{2+\kappa} \min\{ \left\langle v \right\rangle^{\gamma}, n^{\gamma} \} \,\rd F_{\tau}^{n}(v) \right) \,\rd \tau\\
				\lesssim & R^{-\kappa}.
			\end{split}
		\end{equation}
	By the similar estimate, we also have
		\begin{equation}\notag
			I_{2} \leq \left|  \int_{0}^{t}\int_{\bR^{3}}\int_{\bR^{3}} \chi_{\{|v|> R\} \cup \{|v_{*}|> R\}} \Psi^{e}(v,v_{*})  \,\rd F_{\tau}(v) \,\rd F_{\tau}(v_{*}) \,\rd \tau \right| \lesssim R^{-\kappa}.
		\end{equation}
		On the other hand, the third term $ I_{3} $ of \eqref{symmweak} converges to zero uniformly for $ t \geq 0 $, as $ n\rightarrow \infty $, since the function $ \Psi^{e}_{n}(v,v_{*}) $ converges to $ \Psi^{e}(v,v_{*}) $ uniformly on a compact set of $ (v,v_{*}) \in  \mathbb{R}_{v}^{3} \times \mathbb{R}_{v_*}^{3} $. Moreover, by Lemma \ref{DissiLemma} (iii), the fourth term $ I_{4} $ of \eqref{symmweak} also converges to zero uniformly on a compact set of $ t \in [0,\infty) $, as $ n \rightarrow \infty $. Then by taking $R\to \infty,$ we have $I_1-I_2+I_3+I_4 \to 0$ and we obtain that $F_t$ is indeed a measure-valued solution.  
		This completes the proof of Theorem \ref{main1}.
	\end{proof}
	
	\begin{remark}As we observed above, the proof of our main theorem {\rm(Theorem \ref{main1})} relies on the fact that $\kappa >0$ and this cannot be relaxed to $\kappa=0$. Note that since we assume the initial datum $F_0$ has finite $(2+\kappa)$-moment, the case $\kappa=0$ corresponds to  the energy bound $F_0\in P_2(\mathbb{R}^3).$  In order to relax our assumption to the sharp energy bound, we believe that one must establish a more improved Povzner estimate than our refined Povzner estimate {\rm(Proposition \ref{Povzner})}. 	\end{remark}
	
	This completes the proof of the global existence of a measure-valued solution in the case without angular cutoff. In the next section, we prove that the solution produces additional moments, which are also bounded.
	
	\section{Moment-creation property}
	\label{sec:moment}
	In this section, we will present the proof of our main theorem (Theorem~\ref{main2}) on the moments production of the corresponding measure-valued solution in Theorem \ref{main1}.
	
	\begin{proof}[Proof of Theorem \ref{main2} ]
		We first introduce the outline of the proof. We will basically apply the inductive strategy to prove the moments-creation property of the measure-valued solution. During the induction process, we will prove that for each iteration, the solution gains an additional $ \gamma $-order moments in an arbitrarily small time interval such that any higher order moments will instantly become finite after the time evolution. 
		
		Motivated by \cite[pp. 479]{MW1999}, to achieve the moments-estimate in each time interval, we will first try to construct the convex approximation $\psi_{\kappa,m}(x)$ of the weight function $ \psi_2(x) = (1+x)^{1+\frac{\kappa}{2}} - 1$ with $ \kappa > 0 $. However, in our case, we remark that the collision process is inelastic and it conserves the total mass and momentum but not the energy as in \eqref{conservQe} and \eqref{disspationQe}. Therefore, we have to design a different type of approximating functions $p_{\kappa,m}$ defined as follows:
			\begin{equation}\label{Psi1}
			\psi_{\kappa,m}(x) \eqdef \begin{cases}
				\psi_2(x) ,& \text{if}\ x \leq m, \\
				p_{\kappa,m}(x),& \text{if}\ x> m,
			\end{cases}
		\end{equation}
		where
		\begin{equation}\label{pp}
			p_{\kappa,m}(x) = C_{\kappa}(m) x + \psi_2(m) -  C_{\kappa}(m) m
		\end{equation}
		for $ \kappa > 0 $ and all $ m \in \mathbb{N} $ where $C_\kappa(x)$ is defined as the derivative of $\psi_2(x)$. 
    Then, if we further define 
		\begin{multline}\label{pkm}
			\psi^{\kappa}_{m}(|v|^{2}) \eqdef \psi_{\kappa,m}(|v|^{2}) - p_{\kappa,m}(|v|^{2}) \\=\begin{cases}
				\psi_2(|v|^{2}) - \psi_2(m) - C_{\kappa}(m) \left[|v|^{2}-m\right] ,& \text{if}\ |v|^2 \leq m, \\
				0,& \text{if}\ |v|^2 > m,
			\end{cases}
		\end{multline}
		where
		$\psi^\kappa_m$ is a bounded and continuously differentiable approximation of $\psi_\kappa$ with compact support, where $ \psi^{\kappa}_{m}(v) \in C^{1}_{0}(\mathbb{R}^{3}) $, and $\partial \psi^{\kappa}_{m}(v) \in C_{0}(\mathbb{R}^{3}) $ is Lipschitz continuous with a uniform Lipschitz constant. 
				
		Then, we observe that,
		\begin{multline}\label{p1}
		    	\int_{\bR^{3}} \psi_{\kappa,m}(|v|^{2}) \,\rd F_{t}(v) - \int_{\bR^{3}} \psi_{\kappa,m}(|v|^{2}) \,\rd F_{0}(v)\\
		= \int_{\bR^{3}} \psi_{m}^{\kappa}(|v|^{2}) \,\rd F_{t}(v) - \int_{\bR^{3}} \psi_{m}^{\kappa}(|v|^{2}) \,\rd F_{0}(v) \\+ \int_{\bR^{3}} p_{\kappa,m}(|v|^{2}) \,\rd F_{t}(v) - \int_{\bR^{3}} p_{\kappa,m}(|v|^{2}) \,\rd F_{0}(v),
		\end{multline} 
            where, by noting \eqref{pp},  the last two terms in the equation above can be further deduced as follows
            \begin{multline}
		   \int_{\bR^{3}} p_{\kappa,m}(|v|^{2}) \,\rd F_{t}(v) - \int_{\bR^{3}} p_{\kappa,m}(|v|^{2}) \,\rd F_{0}(v)\\
      = C_{\kappa}(m)\int_{\bR^{3}} |v|^{2} \,\rd F_{t}(v) - C_{\kappa}(m)\int_{\bR^{3}} |v|^{2} \,\rd F_{0}(v)  \\
      + \left[ \int_{\bR^{3}} \left(\psi_2(m) -  C_{\kappa}(m) m \right) \,\rd F_{t}(v) - \int_{\bR^{3}} \left( \psi_2(m) -  C_{\kappa}(m) m\right) \,\rd F_{0}(v) \right] \\
      = C_{\kappa}(m)\int_{\bR^{3}} |v|^{2} \,\rd F_{t}(v) - C_{\kappa}(m)\int_{\bR^{3}} |v|^{2} \,\rd F_{0}(v)
		\end{multline} 
        where the vanishment of the last terms in the second equality above is due to the conservation of mass. Hence, \eqref{p1} becomes
		\begin{multline*}
			\int_{\bR^{3}} \psi_{\kappa,m}(|v|^{2}) \,\rd F_{t}(v) + C_{k}(m)\int_{\bR^{3}} |v|^{2} \,\rd F_{0}(v) \\= \int_{\bR^{3}} \psi_{m}^{\kappa}(|v|^{2}) \,\rd F_{t}(v) - \int_{\bR^{3}} \psi_{m}^{\kappa}(|v|^{2}) \,\rd F_{0}(v) +C_{k}(m)\int_{\bR^{3}} |v|^{2} \,\rd F_{t}(v).
		\end{multline*}
		After substituting $\psi=\psi^\kappa_m$ by approximating $C^2_b$ functions into the weak formulation \eqref{weak}, we find, for $ t \geq 0 $,
		\begin{equation}\label{M1}
			\begin{split}
				&\int_{\bR^{3}} \psi^\kappa_m(|v|^{2}) \,\rd F_{t}(v) - \int_{\bR^{3}} \psi^\kappa_m(|v|^{2}) \,\rd F_{0}(v)\\
				&= \frac{1}{2} \int_{0}^{t}\,\rd\tau \,\int_{\bR^{3}}\rd F_{\tau}(v_{*}) \,\int_{\bR^{3}}\rd F_{\tau}(v) \, |v-v_{*}|^{\gamma} \\
				& \qquad \qquad \quad \qquad \times \int_{\Sd^{2}} b(\hat{v}_-\cdot\sigma)  \left[ \psi^\kappa_m(|v'|^{2}) + \psi^\kappa_m(|v'_{*}|^{2}) - \psi^\kappa_m(|v|^{2}) - \psi^\kappa_m(|v_{*}|^{2}) \right] \,\rd \sigma.
			\end{split}
		\end{equation}
    By noting $\psi_m^{\kappa} = \psi_{\kappa,m} - p_{\kappa,m}$, we further have
    \begin{equation}\label{M2}
    \begin{split}
        &\int_{\bR^{3}} \psi^\kappa_m(|v|^{2}) \,\rd F_{t}(v) - \int_{\bR^{3}} \psi^\kappa_m(|v|^{2}) \,\rd F_{0}(v)\\
        =&  \frac{1}{2} \int_{0}^{t}\,\rd\tau \,\int_{\bR^{3}}\rd F_{\tau}(v_{*}) \,\int_{\bR^{3}}\rd F_{\tau}(v) \, |v-v_{*}|^{\gamma} \\
				&\qquad \qquad \times\int_{\Sd^{2}} b(\hat{v}_-\cdot\sigma)  \left[ \psi_{\kappa,m}(|v'|^{2}) + \psi_{\kappa,m}(|v'_{*}|^{2}) - \psi_{\kappa,m}(|v|^{2}) - \psi_{\kappa,m}(|v_{*}|^{2}) \right] \,\rd \sigma\\
				&\qquad - \frac{1}{2} \int_{0}^{t}\,\rd\tau \,\int_{\bR^{3}}\rd F_{\tau}(v_{*}) \,\int_{\bR^{3}}\rd F_{\tau}(v) \, |v-v_{*}|^{\gamma} \\
				&\qquad \qquad \times\int_{\Sd^{2}} b(\hat{v}_-\cdot\sigma)  \left[ p_{\kappa,m}(|v'|^{2}) + p_{\kappa,m}(|v'_{*}|^{2}) - p_{\kappa,m}(|v|^{2}) - p_{\kappa,m}(|v_{*}|^{2}) \right] \,\rd \sigma.
    \end{split}
    \end{equation}
		For the last term in \eqref{M2} above, we use \eqref{pp} and the collision invariance of $ 1 $ and obtain
		\begin{equation}\label{p2}
		\begin{split}
			&\frac{1}{2} \int_{0}^{t}\,\rd\tau \,\int_{\bR^{3}}\rd F_{\tau}(v_{*}) \,\int_{\bR^{3}}\rd F_{\tau}(v) \, |v-v_{*}|^{\gamma} \\
			&\qquad \qquad \times\int_{\Sd^{2}} b(\hat{v}_-\cdot\sigma)  \left[ p_{\kappa,m}(|v'|^{2}) + p_{\kappa,m}(|v'_{*}|^{2}) - p_{\kappa,m}(|v|^{2}) - p_{\kappa,m}(|v_{*}|^{2}) \right] \,\rd \sigma \\
			&= \frac{1}{2}C_{\kappa}(m)  \int_{0}^{t}\,\rd\tau \,\int_{\bR^{3}}\rd F_{\tau}(v_{*}) \,\int_{\bR^{3}}\rd F_{\tau}(v) \, |v-v_{*}|^{\gamma} \\
			&\qquad \qquad \qquad \qquad\qquad \qquad\qquad \qquad\times\int_{\Sd^{2}} b(\hat{v}_-\cdot\sigma)  \left[ |v'|^{2} + |v'_{*}|^{2} - |v|^{2} - |v_{*}|^{2} \right] \,\rd \sigma\\
   &=-\frac{1}{2}C_{\kappa}(m)  \int_{0}^{t}\,\rd\tau \,\int_{\bR^{3}}\rd F_{\tau}(v_{*}) \,\int_{\bR^{3}}\rd F_{\tau}(v) \, |v-v_{*}|^{\gamma} \\
			&\qquad \qquad \qquad \qquad\qquad \qquad\qquad \qquad\times\int_{\Sd^{2}} b(\hat{v}_-\cdot\sigma)\frac{1-e^2}{2}\frac{1-(\hat{v}_- \cdot \sigma)}{2}|v-v_*|^2\,\rd \sigma,
		\end{split}
		\end{equation} by \eqref{energyloss}. 
		Then the right-hand side of \eqref{p2} can also be well-estimated by the Povzner inequality in Proposition~\ref{Povzner}.
	
		By combing \eqref{p1}-\eqref{p2}, we obtain,
		\begin{equation}\notag
			\begin{split}
					&\int_{\bR^{3}} \psi_{\kappa,m}(|v|^{2}) \,\rd F_{t}(v) + C_{k}(m)\int_{\bR^{3}} |v|^{2} \,\rd F_{0}(v) \\
					&= \frac{1}{2} \int_{0}^{t} \int_{\bR^{3}} \int_{\bR^{3}} |v-v_{*}|^{\gamma}  \Big[ H^{\kappa}_{m}(v,v_{*}) + G^{\kappa}_{m}(v,v_{*}) \Big] \rd F_{\tau}(v) \rd F_{\tau}(v_{*}) \,\rd\tau\\
					& \qquad -\frac{C_{\kappa}(m)}{2} \int_{0}^{t}\,\rd\tau \,\int_{\bR^{3}}\rd F_{\tau}(v_{*}) \,\int_{\bR^{3}}\rd F_{\tau}(v) \, |v-v_{*}|^{\gamma} \\
					&\qquad \qquad \times\int_{\Sd^{2}} b(\hat{v}_-\cdot\sigma)\frac{1-e^2}{2}\frac{1-(\hat{v}_- \cdot \sigma)}{2}|v-v_*|^2\,\rd \sigma\\
					&\qquad  + C_{k}(m)\int_{\bR^{3}} |v|^{2} \,\rd F_{t}(v),
			\end{split}
		\end{equation}
		where $ H^{\kappa}_{m}(v,v_{*}), G^{\kappa}_{m}(v,v_{*}) $ are defined in the same way as in Proposition \ref{Povzner} except for now replacing $ \psi_2 $ by $ \psi_{\kappa,m} $ and changing $ b_{n} $ into $ b $.

  Then, by applying the Povzner estimate in Proposition \ref{Povzner} onto $H^\kappa_m$ and $G^\kappa_m$ and by further letting $ m \rightarrow \infty $ and noticing $ C_{\kappa}(m) $ will approach to $ 1 $, we find that, for suitable $ C_1, C_2, C_3, C_e > 0 $ and arbitrarily small $ t_{0} > 0 $,
		\begin{multline}\label{t0}
				\int_{\bR^{3}} \left\langle v \right\rangle^{2+\kappa} \,\rd F_{t_{0}}(v) +\int_{\bR^{3}} | v |^{2} \,\rd F_{0}(v)\\
    + C_1\int_{0}^{t_0} \left( \int_{\bR^{3}} \left\langle v \right\rangle^{2+\kappa+\gamma}  \,\rd F_{\tau}^{n}(v) \right) \left( \int_{\bR^{3}} \left\langle v_* \right\rangle^{2+\gamma}  \,\rd F_{\tau}^{n}(v_*) \right) \,\rd \tau \\
				\leq    \int_{\bR^{3}} |v|^{2} \,\rd F_{t_{0}}(v) + C_3 \int_{0}^{t_{0}}\int_{\bR^{3}} \left\langle v \right\rangle^{2+\gamma}  \,\rd F_{\tau}(v) \,\rd \tau \\+C_e\int_{0}^{t_0}\,\rd\tau \left( \int_{\bR^{3}} \left\langle v \right\rangle^{2+\gamma}  \,\rd F_{\tau}^{n}(v) \right) \left( \int_{\bR^{3}} \left\langle v_* \right\rangle^{2+\gamma}  \,\rd F_{\tau}^{n}(v_*) \right)\\
        + C_2\int_{0}^{t_0} \left( \int_{\bR^{3}} \left\langle v \right\rangle^{1+\kappa+\gamma}  \,\rd F_{\tau}^{n}(v) \right) \left( \int_{\bR^{3}} \left\langle v_* \right\rangle^{1+\gamma}  \,\rd F_{\tau}^{n}(v_*) \right) \,\rd \tau,
		\end{multline}
		where the last three terms on the right-hand side are finite by  Corollary~\ref{Proposition2+kappa.noncut}.
		As the second term of the left-hand side \eqref{t0} is finite, it then implies that, there exists $ t_{1} \in (0, t_{0}) $ such that,
		\begin{equation}\label{t1}
			\int_{\bR^{3}} \left\langle v \right\rangle^{2+\kappa_{1}} \,\rd F_{t_{1}} (v) < \infty,
		\end{equation}
		where we denote $ \kappa_{1} = \kappa + \gamma $.
		By repeating the same procedures as above and replacing $ \kappa $ and $ t=0 $ by $ \kappa_{1} $ and $ t_{1} $ respectively, we have, for suitable $ C'_1,C'_2,C'_3, C'_e> 0 $ and $ t \in (t_{1},t_{0}) $,
		\begin{multline}\notag 
				\int_{\bR^{3}} \left\langle v \right\rangle^{2+\kappa_1} \,\rd F_{t_{0}}(v) +\int_{\bR^{3}} | v |^{2+\gamma} \,\rd F_{t_1}(v)\\
    + C'_1\int_{t_1}^{t_0} \left( \int_{\bR^{3}} \left\langle v \right\rangle^{2+\kappa_1+\gamma}  \,\rd F_{\tau}^{n}(v) \right) \left( \int_{\bR^{3}} \left\langle v_* \right\rangle^{2+2\gamma}  \,\rd F_{\tau}^{n}(v_*) \right) \,\rd \tau \\
				\leq    \int_{\bR^{3}} |v|^{2+\gamma} \,\rd F_{t_{0}}(v) + C'_3 \int_{t_1}^{t_{0}}\int_{\bR^{3}} \left\langle v \right\rangle^{2+2\gamma}  \,\rd F_{\tau}(v) \,\rd \tau \\+C'_e\int_{t_1}^{t_0}\,\rd\tau \left( \int_{\bR^{3}} \left\langle v \right\rangle^{2+2\gamma}  \,\rd F_{\tau}^{n}(v) \right) \left( \int_{\bR^{3}} \left\langle v_* \right\rangle^{2+2\gamma}  \,\rd F_{\tau}^{n}(v_*) \right)\\
        + C'_2\int_{t_1}^{t_0} \left( \int_{\bR^{3}} \left\langle v \right\rangle^{1+\kappa_1+\gamma}  \,\rd F_{\tau}^{n}(v) \right) \left( \int_{\bR^{3}} \left\langle v_* \right\rangle^{1+2\gamma}  \,\rd F_{\tau}^{n}(v_*) \right) \,\rd \tau.
		\end{multline}
	Here the right-hand side is also finite by noticing \eqref{t1}. 
	Applying the induction by setting $ \kappa_{j+1} = \kappa_{j} + \gamma = \gamma(j+1) + \kappa $, we obtain that, for any $ T > t_{0} $ and $ l > 0 $, there exists a constant $ C_{T,l} $ such that,
		\begin{equation*}
			\begin{split}
				\sup_{t_{0} \leq t \leq T} \int_{\bR^{3}} \left\langle v \right\rangle^{l} \,\rd F_{t}(v) \leq C_{T,l}.
			\end{split}
		\end{equation*}
		Then we follow the same continuation argument as in  \cite[Lemma 3.8]{LuMouhot2012} or \cite[Proposition 1.4]{morimoto2016measure}  to extend to the infinite time interval $t\ge t_0$ and conclude that $$ \sup_{t \geq t_{0}} M_{l}(t) < \infty. $$ This completes the proof of Theorem \ref{main2}.
	\end{proof}
	
	
	
	\section*{Acknowledgement}
	\label{sec:ack}
	J.~W.~Jang is supported by the National Research Foundation of Korea (NRF) grant funded by the Korean government (MSIT) NRF-2022R1G1A1009044 and by the Basic Science Research Institute Fund of Korea NRF-2021R1A6A1A10042944. 
	K.~Qi is supported by grants from the School of Mathematics at the University of Minnesota, and would like to thank Prof.~Tong Yang for his introducing and discussing the related topic. 
	
	\appendix
        \section{Priliminary Lemma for Convex Functions}
        For the functions $\psi$ satisfying the following list of conditions:
        \begin{equation}\label{condition_psi}
        \begin{split}
            \psi(x) \geq&\  0, \quad x > 0; \  \psi(0) = 0 \\
            \psi(x) \in C^{1}([0,\infty)) &\ \text{is convex}, \quad \psi''(x) \ \text{is locally bounded};\\
            \psi'(ax) \leq&\ \eta_1(a) \psi'(x), \ x >0, \  a > 1;\\
            \psi''(ax) \leq&\ \eta_2(a) \psi''(x), \ x > 0, \  a > 1
        \end{split}
        \end{equation}
        where $\eta_1(a)$ and $\eta_2(a)$ are functions of $a$ only, bounded on every finite interval $a>0$.

        \begin{lemma}{\cite[Lemma 3.1]{GambaDiffusively}}
            Assume that $\psi(x)$ satisfies conditions \eqref{condition_psi}. Then, 
            \begin{equation}\label{A2}
                \psi(x+y) - \psi(x) - \psi(y) \leq \tilde{A} (x\psi'(y) + y \psi'(x))
            \end{equation}
            and
            \begin{equation}\label{A3}
                \psi(x+y) - \psi(x) - \psi(y) \geq \tilde{b} xy \psi''(x+y)
            \end{equation}
            where $\tilde{A} = \eta_1(2)$ and $\tilde{b} = (2\eta_2(2))^{-1}$.
        \end{lemma}

	\bibliographystyle{amsplain}
	\bibliography{Bib_Kunlun_inelastic}
	
\end{document}